\documentclass{amsart}

\usepackage[a4paper,margin=3.06cm]{geometry}

\usepackage{amssymb}
\usepackage{amsmath}
\usepackage{amsthm}
\usepackage{cite}
\usepackage{enumerate}
\usepackage{comment}
\usepackage{xspace}
\usepackage[retainorgcmds]{IEEEtrantools}
\usepackage{xcolor}
\usepackage{tikz}
\usepackage{tikz-cd}
\usepackage{hyperref}
\usetikzlibrary{patterns}
\usepackage{bbm}

\theoremstyle{plain}
\newtheorem{Thm}{Theorem}[section]
\newtheorem{Cor}[Thm]{Corollary}
\newtheorem{Lemma}[Thm]{Lemma}
\newtheorem{Obs}[Thm]{Observation}
\newtheorem{Prop}[Thm]{Proposition}

\theoremstyle{definition}
\newtheorem{Def}[Thm]{Definition}

\newtheorem{Rmk}[Thm]{Remark}
\newtheorem{Question}[Thm]{Question}

\theoremstyle{remark}
\newtheorem{Claim}[Thm]{Claim}

\numberwithin{equation}{section}

\makeatletter
\newcommand\xleftrightarrow[2][]{%
  \ext@arrow 9999{\longleftrightarrowfill@}{#1}{#2}}
\newcommand\longleftrightarrowfill@{%
  \arrowfill@\leftarrow\relbar\rightarrow}
\makeatother

\newcommand{\R}{{\mathbb{R}}}

\newcommand{\N}{{\mathbb{N}}}
\newcommand{\Z}{{\mathbb{Z}}}
\newcommand{\C}{{\mathbb{C}}}

\newcommand{\PP}{{\mathcal{P}}}

\newcommand{\csalg}{$C^*$-algebra\xspace}
\newcommand{\csalgs}{$C^*$-algebras\xspace}

\newcommand{\vnalg}{von Neumann algebra\xspace}
\newcommand{\vnalgs}{von Neumann algebras\xspace}

\newcommand{\Ad}{\mathrm{Ad}}
\newcommand{\Matrix}{\mathbf{M}}
\newcommand{\Mn}{\Matrix_n}
\newcommand{\MnMn}{\Mn*\Mn}
\newcommand{\BH}{\mathcal{B}(\mathcal{H})}

\newcommand{\Hil}{\mathcal{H}}

\newcommand{\twoone}{\mathrm{II}_1}
\newcommand{\tr}{\mathrm{tr}}

\newcommand{\eps}{\varepsilon}

\newcommand{\restrict}[1]{\vert_{#1}}

\newcommand{\fromto}[2]{\colon #1 \rightarrow #2}

\newcommand{\Traces}[1]{\mathrm{T}\left(#1\right)}
\newcommand{\TracesP}[2]{\mathrm{T}_{#2}\left(#1\right)}

\newcommand{\Tr}[1]{\mathrm{Tr}\left(#1\right)}
\newcommand{\Ch}[1]{\mathrm{Ch}\left(#1\right)}

\newcommand{\relCh}[2]{\mathrm{Ch}_{#1}\left(#2\right)}
\newcommand{\relTr}[2]{\mathrm{Tr}_{#1}\left(#2\right)}

\title[The Space of Traces of  Free Products ]{The Space of Traces of the Free Group and Free Products of Matrix Algebras}
\author{Joav Orovitz, Raz Slutsky, Itamar Vigdorovich}
\date{}

\begin{document}

\maketitle

\begin{abstract}
We show that the space of traces of the free group $F_d$ on $2\leq d \leq \infty $ generators is a Poulsen simplex, i.e., every trace is a pointwise limit of extreme traces. This fails for many virtually free groups. The same result holds for free products of the form $C(X_1)*C(X_2)$ where $X_1$ and $X_2$ are compact metrizable spaces without isolated points. 
Using a similar strategy, we show that the space of traces of  the free product of matrix algebras $\Mn(\C) * \Mn(\C)$ is a Poulsen simplex as well, answering a question of Musat and R\o rdam for $n \geq 4$. Similar results are shown for certain faces of the simplices above, such as the face of finite-dimensional traces or amenable traces. 
\end{abstract}

\section{Introduction}
A \emph{trace} on a group $G$ is a 
 function $\varphi: G \rightarrow \C$ which is positive definite, conjugation invariant, and normalized so that $\varphi(e) = 1$. Traces which cannot be written as a proper convex combination of other traces are called \emph{characters}. 
The space of traces of $G$ is a compact convex set with the point-wise convergence topology.

Traces on groups have been widely studied in different names and contexts. For example, the space of characters of $G$ is equal to the  Pontryagin dual when $G$ is abelian and with the set of irreducible representations (up to equivalence) when $G$ is finite.
It is often referred to as the Thoma dual of $G$,
\cite{bekka2020unitary}. Notably, characters served as a fundamental tool in Thoma's celebrated classification of type I discrete groups \cite{thoma1964unitare}. It is thus of general interest to describe the Thoma dual of groups, a task which was achieved for several interesting classes of groups \cite{thoma1964unzerlegbaren,pukanszky1974characters,howe1977representations,bekka2007operator,bekka2020characters}.

The importance of characters goes beyond harmonic analysis, and in recent years their study has found remarkable applications in the structure of normal subgroups, unitary representations, hyperlinearity, stability, invariant random subgroups, uniform recurrent subgroups, as well as $C^*$-simplicity and the unique trace property\cite{bekka2007operator,boutonnet2021stationary,bader2022charmenability,hadwin2018stability,levit2022characters,hartman2017stationary,breuillard2017c}.

More generally, for  a separable unital \csalg $A$, we have the set  $\Traces{A}$ of all tracial states on $A$ with the weak-* topology. When $A=C^*(G)$ is the maximal \csalg of a group $G$, the tracial states on $A$ are in one-to-one correspondence with traces on the group $G$.
 The space $\Traces{A}$ carries a lot of information about the algebra and has long been known to be a very useful invariant of it. For example, it plays an important role in the attempt to classify \csalgs as part of Elliott's program \cite{rordam2002classification,winter2018structure,strung2021introduction},  and the recently settled Connes' Embedding Problem \cite{brown06, ozawa2013tsirelson, mip=re}.\\
 
One important property of the space of traces of a countable group, or of a separable unital \csalg, is that it is a metrizable Choquet simplex \cite[3.1.18]{sakai2012c}, that is, a compact convex set such that every point in it is the barycenter of a unique Borel probability measure supported on the set of  extreme points. We refer to \cite{phelps2001lectures} for an overview of Choquet theory.

In \cite{Poulsen},  Poulsen constructed a simplex with the counter-intuitive property that its set of extreme points is dense. A simplex with this property is called a \emph{Poulsen simplex}. Later, it was shown in \cite{los_78} that there is a unique Poulsen simplex; namely,  all  infinite metrizable Poulsen simplices are affinely homeomorphic. This simplex is universal in the sense that it admits any metrizable Choquet simplex as a face. Together with a homogeneity condition (see \cite{los_78}), this universal property characterizes the Poulsen simplex. Moreover, its set of extreme points is homeomorphic to a separable Banach space, e.g. $\ell^2$.

\subsection*{The free group}
The first result of this paper concerns free products of abelian groups
\begin{Thm}
\label{Main Theorem}
Let $G_1$ and $G_2$ be infinite countable abelian groups, and let $G_0$ be any countable group. Then the space of traces on $G=G_0*G_1*G_2$ is a Poulsen simplex. 
\end{Thm}
Theorem \ref{Main Theorem} is an immediate consequence of  Theorem \ref{Thm:free products of commutative algebras} below. 
The most notable example is the free group $F_d$.
\begin{Cor} \label{Cor Free Group}
    The space of traces of a free group $F_d$ on $2\leq d\leq \infty $ generators is a Pouslen simplex.
\end{Cor}

Groups for which such results are far from being true are property (T) groups; indeed, the space of traces of such groups is a Bauer simplex, meaning that the set of extreme points is closed \cite{kennedy2022noncommutative,levit2023characterlimits}.  Moreover, the group $\operatorname{SL}_n(\Z)$ for $n\geq 3$, and more generally, lattices in simple Lie groups of rank at least $2$ are character-rigid, in the sense that aside from the Dirac character $\delta_e$ (assuming trivial centre for simplicity), any other character is of the form $\frac{1}{\dim\pi}\tr\circ \pi $ for some finite-dimensional representation $\pi$ \cite{bekka2007operator,peterson,boutonnet2021stationary} (see \cite{peterson2016character, bader2021charmenability, bader2022charmenability, bader-iti,lavi2023characters} for generalizations). This rigidity generalizes, for example, the Margulis normal subgroup theorem. 

When the real rank is $1$, character rigidity fails. Since any free group $F_d$ for $2\leq d <\infty$ is a lattice inside $\mathrm{SL}_2(\mathbb{R})$, Corollary \ref{Cor Free Group} may be viewed as a very strong negation of character rigidity. It moreover provides some description of the space of characters (topologically, at least) in a realm where, as pointed out in \cite[Remark 4]{bekka2007operator}, a true classification cannot be expected. 
More generally, many lattices in rank $1$ semi-simple Lie group surject onto the free group (see \cite{lubotzky1996free}) and therefore admit a closed face which is the Poulsen simplex. We get, in particular, the following corollary for fundamental groups of hyperbolic surfaces:
\begin{Cor}\label{cor: hyperbolic surfaces admit poulsen face}
  If $G$ is the fundamental group of a hyperbolic surface of finite volume, then the space of traces on $G$ admits a face which is a Poulsen simplex. 
\end{Cor}
We leave open the question of whether the space of traces of a surface group is in itself a Poulsen simplex. However,  we do show that the space of traces of many virtually free groups is not a Poulsen simplex. 
\begin{Thm}[See Theorem \ref{thm:free product of ET is not Poulsen} for a generalization] \label{thm:intro amalgam not poulsen}
    The space of traces of an amalgamated free product of finite non-trivial groups is not a Poulsen simplex.
\end{Thm}
This shows for example that the trace space of the groups $\mathrm{SL}_2(\Z)$ and $\mathrm{PSL}_2(\Z)$ is not a Poulsen simplex.\footnote{
This is because  $\mathrm{SL}_2(\Z) \cong \Z / 4\Z *_{\Z / 2 \Z} \Z / 6\Z$ and $\mathrm{PSL}_2(\Z) \cong \Z / 2 \Z * \Z / 3\Z$.
}
The statement of Theorem \ref{thm:intro amalgam not poulsen} holds more generally for amalgamated products of groups with Kazhdan's Property (T). In \S\ref{sec:obstructions} we define a property we call Property (ET), which  is a natural generalization of Property (T), but in contrast, it \emph{is} preserved when taking arbitrary amalgamated products, and in addition, it is \emph{not} preserved when taking finite index subgroups. This property is employed to formulate obstructions preventing groups from having a Poulsen trace space, such as the obstruction given in Theorem \ref{thm:intro amalgam not poulsen}.\\

Certain similar results to Corollary \ref{Cor Free Group} were established previously.
In \cite{Yoshizawa}, the compact convex set of normalized positive definite functions on the free group $F_d$ is studied, and it is shown that the extreme points are dense. Note that this set coincides with the set of states on $C^*(F_d)$, which is \emph{not} a simplex. See also \cite{choi80, lubotzky_shalom}.
Another related result concerns invariant random subgroups (IRSs) coined in \cite{abert2014kesten}; see \cite{gelander2015lecture} for an introduction. In \cite{bowen15}, it is shown that the space of IRSs of the free group which have no finite index subgroups in their support is a Poulsen simplex.
Every IRS on a discrete group gives rise to a trace on it in a natural way, and so Corollary \ref{Cor Free Group} shows that some  properties of the space of IRSs are retained by the space of traces on the free group.

Finally, we note that Corollary \ref{Cor Free Group} implies the same result for the simplex of traces of $C^*(F_n)$. It is known that the maximal $C^*$-algebras of free groups on a different number of generators are non-isomorphic, as they have different $K$-theories \cite{cuntz1982k}; Corollary \ref{Cor Free Group} shows on the other hand that they cannot be distinguished by their trace simplex invariant. 

\subsection*{Free products of commutative \csalgs} For any two unital \csalgs $A_1$ and $A_2$ one can consider their unital universal  free product denoted\footnote{This is often denoted in the literature by $A_1*_\C A_2$} by $A=A_1*A_2$. This is the co-product in the category of unital \csalgs, as defined in Chapter 1 of \cite{voiculescu1992free}. 
\begin{Thm} \label{Thm:free products of commutative algebras}
    Let $X_1$ and $X_2$ be two compact metrizable spaces with no isolated points and let $A_1 = C(X_1), A_2 = C(X_2)$ be the corresponding \csalgs of continuous functions. Consider the free product $A=A_0*A_1*A_2$ where $A_0$ is any unital separable \csalg. Then the space of traces of $A$ is a Poulsen simplex.
\end{Thm}
The proof of Theorem \ref{Thm:free products of commutative algebras} is given in \S\ref{sec:free group}. We note that for groups $G_1$ and $G_2$, it holds that $C^*(G_1)*C^*(G_2)=C^*(G_1*G_2)$ \cite[Proposition 1.4.3]{voiculescu1992free}.
Theorem \ref{Main Theorem} therefore follows from Theorem \ref{Thm:free products of commutative algebras} since the \csalg of an abelian group is the algebra of continuous functions on its Pontryagin dual. In the case where $G_1$ and $G_2$ are infinite discrete abelian, then $A_1=C(\hat{G}_1)$, $A_2=C(\widehat{G}_2)$ both satisfy the assumptions in Theorem \ref{Thm:free products of commutative algebras}.

We note that contrary to group \csalgs or to commutative \csalgs, $A_0$ may admit no traces at all. In such case the space of traces of $A$ is empty, and may be regarded as a Poulsen simplex in a vacuous sense. In any other case, this simplex is infinite.

\subsection*{Free products of matrix algebras}
It is interesting to consider also free products of non-commutative \csalgs, such as the algebra of $n$-by-$n$ complex matrices $\Mn=\Mn(\C)$. Note that while a commutative \csalgs admits an abundance of traces, $\Mn$ has a unique trace.  

The second part of this paper, therefore, deals with the trace simplex of the free product $C^*$-algebra $\MnMn$. This algebra was studied extensively. It was shown to be semi-projective in \cite{blackadar1985shape} and residually finite-dimensional in \cite{exel1992finite}. 
Most recently, Musat and R\o rdam \cite{musat2021factorizable} studied this algebra and related its trace simplex to factorizable quantum channels on $\Mn$. In that paper, they relate questions regarding approximation of traces and Connes' embedding problem to factorizable channels, in the spirit of Tsirelson's conjecture and quantum correlation matrices. In addition, they show that the Poulsen simplex is a face of the trace space of $\MnMn$ for $n\geq 3$, and ask whether this trace space is in itself a Poulsen simplex. Our second theorem answers this question positively for $n\geq 4$. 
\begin{Thm}\label{Second Theorem}
The trace simplex of the free product $\Mn(\C) * \Mn(\C)$ is a Poulsen simplex for $n\geq 4$. 
\end{Thm}

The same theorem holds, with essentially the same proof, for the free product of more than $2$ copies of $\Mn$. While in the case $n=3$ some technical difficulties arise in our method, we believe that the result remains true and that a similar strategy could be used. The proof of Theorem \ref{Second Theorem} is given in \S\ref{sec:matrix algebras}. 
The amalgamated product $\Mn*_D\Mn$, where $D\subseteq \Mn$ is the subalgebra of diagonal matrices, plays a special role in the proof. This is inspired by ideas outlined in \cite{musat2021factorizable}. In fact, we prove that the trace space of $\Mn*_D\Mn$ is also a Poulsen simplex; interestingly, this is a consequence of Corollary \ref{Cor Free Group} on the free group - see Corollary \ref{cor:amalgamated poulsen}.\\

Let us now quickly recall the connection Musat and R\o rdam drew between traces on $\MnMn$ and quantum information theory. 
A linear map $T: \Mn(\C) \rightarrow \Mn(\C)$ is called a \textit{quantum channel} if it is unital, completely-positive and trace-preserving. Such maps model communication in quantum information. A \emph{factorization} of $T$ is a tracial von Neumann algebra $(M, \tau)$, along with two unital $*$-homomorphisms $\alpha, \beta: \Mn(\C) \rightarrow M$ such that $T = \beta^* \circ \alpha$; here $\beta^*: M \rightarrow \Mn(\C)$ is the adjoint of $\beta$ defined by $\langle \beta(x), y \rangle_{\tau} = \langle x, \beta^*(y) \rangle_{\operatorname{tr}_n}$ for all $x \in M, y \in \Mn(\C)$. See \cite{haagerup2011factorization} for a thorough introduction.
A quantum channel admitting a factorization is called \emph{factorizable}. These maps were extensively studied by Haagerup and Musat in \cite{haagerup2011factorization}, where they show that there exist non-factorizable channels for every $n \geq 3$, and deduce that the so-called \textit{asymptotic quantum Birkhoff conjecture} is false. In \cite{musat2020non}, factorizable maps which cannot factor through a finite-dimensional algebra are constructed for $n \geq 11$. 

The compact convex space $\mathcal{FM}(n)$ of all factorizable channels on $\Mn(\C)$ is studied in \cite{musat2021factorizable}, and is shown to be parameterized by traces on $\MnMn$.
 In this light, it is interesting to study the implications of Theorem \ref{Second Theorem} on factorizable channels. Let us illustrate one straightforward consequence. Suppose that a quantum channel $T$ admits a factorization through a tracial von Neumann algebra $M$ via maps $\alpha,\beta:\Mn\to M$. Call this factorization   \emph{surjective} if $\left(\alpha(\Mn)\cup \beta(\Mn)\right)''=M$.   
We deduce the following corollary of Theorem \ref{Second Theorem} (see \S\ref{subsec:quantum} for details).

\begin{Cor}\label{cor:quantum-channels}
The space of quantum channels which factorize surjectively through a finite factor is dense in $\mathcal{FM}(n)$. 
\end{Cor}

\subsection*{Finite-dimensional traces and other faces}
Theorems \ref{Thm:free products of commutative algebras} and \ref{Second Theorem} concern the approximation of traces by extreme traces. A close look at the proofs shows that the constructed approximating extreme traces share many properties with the trace that they are meant to approximate. In \S\ref{sec:faces} we consider certain faces of the space of traces and show that they are Poulsen simplices as well. In particular, we obtain the following:
\begin{Thm}\label{thm:intro-faces-poulsen}
    Let $A=A_1*A_2$ where either $A_1$ and $A_2$ are \csalgs of continuous functions on a compact metrizable space without isolated points, or $A_1=A_2=\Mn$ with $n\geq 4$. Consider the (Poulsen) simplex of traces $\Traces{A}$. Then the closure of the set of all finite-dimensional traces is a face of $\Traces{A}$ and is a Poulsen simplex as well.  The same is true if `finite-dimensional' is replaced with `amenable', `uniformly amenable' or `Property (T)'. 
\end{Thm}
Other properties which are covered are McDuff, Haagerup, Full, Properly proximal, Gamma, and more. We refer to \S\ref{sec:faces} for precise definitions and statements including Corollary \ref{cor:Poulsen-faces} which is a generalization of Theorem \ref{thm:intro-faces-poulsen}.

Specifically for the closure of finite-dimensional traces a more algebro-geometric approach can be adopted, and certain stronger statements for this face can be obtained, see Proposition \ref{prop:finite-dimensional amplification}. It in particular follows that the closure of the set of finite-dimensional traces of $\Matrix_3(\C)*\Matrix_3(\C)$ is a Poulsen simplex as well. Such analysis is performed in \S\ref{section:findim}, and includes an amplification of dimension argument which is crucial for the proof of Theorem \ref{Second Theorem}. 

Questions regarding finite-dimensional states and traces have been studied extensively. For example, it was shown in \cite{choi80, exel1992finite} that the finite-dimensional states on the algebras $C^*(F_d)$ and $\MnMn$ are dense in the space of all states.
This remains true for the free group even when restricting to positive definite functions that factor through a finite quotient \cite{lubotzky_shalom}. 
In light of these results, and Theorem \ref{thm:intro-faces-poulsen}, one might wonder if our main results can be improved by showing that the finite-dimensional traces are dense in the space of all traces. The answer to this question is quite interestingly no. For certain algebras, among them are $C^*(F_d)$ and $\MnMn$, the density of the finite-dimensional traces is equivalent to the Connes embedding problem
\cite[Sec. 4.3]{brown06},\cite[Sec. 4.1]{musat2021factorizable}, which was recently refuted in \cite{mip=re}. See also \cite{hadwin2018stability,hadwin2018tracial,musat2020non} for other interesting aspects of the face of finite-dimensional traces and its closure.\\

Going back to the connection between IRSs and traces, the analogous question for IRSs is whether every IRS of the free group is \emph{co-sofic}, that is, whether the IRSs supported on finite index subgroups are weak-* dense in the simplex of all IRSs. This problem is referred to as the Aldous–Lyons conjecture \cite{aldous_lyons}, and it remains open. An affirmative answer for it would imply that all groups are sofic. 

\subsection*{Concluding remarks}
Our results make use of a simple strategy, and use the well known fact that a trace $\varphi$ is extreme if and only if the associated \vnalg $\pi_\varphi (A)^{\prime\prime}$ obtained by the GNS construction is a factor (cf. Theorem \ref{thm:extremal-factor}).
Given a trace $\varphi$ on $A$ which is a convex combination of extreme traces, we first define an embedding of $\pi_\varphi (A)^{\prime\prime}$ into a finite factor $M$ such that $\tau_M \circ \pi_\varphi = \varphi$, where $\tau_M$ is the unique trace on $M$.
We then make small perturbations to the representation $\pi_\varphi$ of $A$ to obtain a representation $\tilde{\pi}:A\to M$, which is close to $\pi_\varphi$ in the trace norm, and such that $\tilde{\pi}(A)^{\prime\prime} = M$.
Most of the work is done towards making the required perturbations to $\pi_\varphi$. 

We note that this strategy is naturally geared toward \csalgs which can be described via some universal property, as this allows a relatively convenient way to define new representations. This leads us to raise the following natural question:

\begin{Question}
Is there a condition on a \csalg $A$ which allows the strategy above to be carried out? 
\end{Question}
Specific examples of interest are free products of finite groups. As a start, we would like to mention the following observation. 
\begin{Obs}\label{cor:free-products-Poulsen-face}
    Consider a free product of groups $G=G_1*G_2$. Assume that both $G_1$ and $G_2$ admit an irreducible   $n$-dimensional 
  unitary representation, for some $n\geq 3$. Then the space of traces of $G$ admits a face which is a Poulsen simplex.
\end{Obs}
Indeed, the unitary representation of each $G_i$ extends to a $*$-representation of $C^*(G_i)$, which surjects onto  $\Mn(\C)$ by Schur's Lemma.
We thus have a surjective $*$-homomorphism $C^*(G) \to \MnMn$, which according to  \cite[Theorem 3.9]{musat2021factorizable} (see also Corollary \ref{cor:amalgamated poulsen}), has a face which is affinely homeomorphic to the Poulsen simplex.

We recall that the trace space of a free product of finite groups is never a Poulsen simplex (Theorem \ref{thm:intro amalgam not poulsen}). We suspect, however, that the reason for this is finite-dimensional in nature, namely, we believe that the following question might have a positive answer:
\begin{Question}
    Let $G=G_1*...*G_d$ be a free product of finite simple groups. Assume either that $d\geq 3$, or that $d\geq 2$ and  the order of $G_1$ is at least $3$. Let $C$ be the  closed convex hull of all extreme infinite-dimensional traces on $G$. Is $C$ a Poulsen simplex?
\end{Question}
A similar question may be asked for the trace simplex of $C(X_1)*C(X_2)$ where $X_1$ and $X_2$ are compact spaces which admit isolated points.

Finally, we raise the following dynamical aspect of our results. Since the group of automorphisms of a group acts on the space of traces, we get,  by Corollary \ref{Cor Free Group},
 a natural action of the automorphism group $\operatorname{Aut}(F_d)$ (or $\operatorname{Out}(F_d)$) on the Poulsen simplex. 
This group has gathered significant attention in the last years and was only recently proven to have property $(T)$, when $d \geq 4$ \cite{kaluba2021property, nitsche2020computer}. 
We note that while the trace space of $F_d$ is a Poulsen simplex (for $d\geq 2$), the space of $\operatorname{Aut}(F_d)$-invariant traces is a Bauer simplex (for $d\geq 4$), namely, the set of extreme points is compact \cite{levit2023characterlimits}.
$\operatorname{Aut}(F_d)$-invariant traces on $F_d$  have gathered interest in recent years with connection to the Wiegold Conjecture and certain questions regarding word measures on groups; see \cite{collins2019automorphism} for a survey. 
The realization of an action of $\operatorname{Aut}(F_d)$ on this universal simplex is therefore interesting towards further understanding the dynamics of $\operatorname{Aut}(F_d)$. We thus raise the following potential extension of Corollary \ref{Cor Free Group} as a starting point:

\begin{Question}
Does there exist an extreme trace on $F_d$ whose $\operatorname{Aut}(F_d)$-orbit is dense in the space of traces?
\end{Question}

\subsection*{Structure of the paper}
Basic facts on traces and von Neumann algebras are surveyed in \S\ref{sec:general statements}. The proof of Theorem \ref{Thm:free products of commutative algebras} on free products of commutative \csalgs is given in \S\ref{sec:free group}, and the proof of Theorem \ref{Second Theorem} on free products of matrix algebras is given in \S\ref{sec:matrix algebras}. The face of finite-dimensional traces is discussed in \S\ref{section:findim} using representation varieties. Other faces are discussed in \S\ref{sec:faces}, and in particular, this section contains the proof of Theorem \ref{thm:intro-faces-poulsen}. We finish with property (ET) and obstructions for having a Poulsen trace simplex in \S\ref{sec:obstructions}, including the proof of Theorem \ref{thm:intro amalgam not poulsen}.

\subsection*{Acknowledgements} 
The authors extend their special gratitude to Adam Dor-On for his invaluable assistance with numerous suggestions, as well as for being a source of great inspiration throughout this work.  The authors also thank Ilan Hirshberg, Mehdi Moradi and Mikael Rørdam for very helpful remarks on the manuscript.  Our appreciation goes to Nir Avni and Elyasheev Leibtag for their useful insights concerning \S\ref{section:findim}, to Arie Levit for engaging discussions related to \S\ref{sec:obstructions}, and to Guy Salomon for advice and guidance. The second and third authors also wish to acknowledge Uri Bader and Tsachik Gelander for their unwavering support.





\section{Preliminaries}\label{sec:general statements}
Throughout the paper, if $a$ is an operator on a Hilbert space, the notation $\|a\|$ will always mean the standard (operator) norm of $a$. Other norms will be indicated with suitable subscripts.
\subsection{Tracial representations}
The proofs of the main results depend crucially on the intimate relation between traces and tracial representations. 

Let $\Hil$ be a Hilbert space, and  denote by $\BH$ the algebra of all bounded operators on $\Hil$.
Let $M\subseteq\BH$ be a \vnalg containing the identity $1_\Hil$. $M$ is called a \emph{factor} if its center is trivial, i.e consists only of the scalar operators  $\C\cdot 1_\Hil$. A map between von Neumann algebras is called \emph{normal} if it is continuous with respect to the ultraweak topology. Morphisms between von Neumann algebras are, by definition, normal $*$-homomorphisms. 
For a subset $S \subseteq M$, $W^*(S)$ denotes the von Neumann subalgebra generated by $S$. 
By the double commutant theorem $W^*(S)=S''\subseteq \BH$.

  By a \emph{tracial von Neumann algebra} we mean a pair $(M,\tau)$ where $M$ is a von Neumann algebra and $\tau$ is a normal and faithful trace on $M$. The trace norm on $M$ is defined by $\|x\|_{\tau}=\tau(x^*x)^{\frac{1}{2}}$ for $x\in M$. A tracial von Neumann algebra which is a factor is called a \emph{finite factor}. A classical result by Murray and von Neumann states that a finite factor admits a unique trace $\tau_M$, and this trace is  normal and faithful \cite{murray1937rings}. 
The most basic example of a finite factor is the algebra of $n$-by-$n$ matrices $\Mn=\Mn(\C)$ -- the unique trace on it is the normalized trace which we denote by $\tr_n$.  More generally, for a von Neumann algebra $N$ we denote by $\Mn(N)$ the von Neumann algebra of $n$-by-$n$ matrices with entries in $N$. This algebra is canonically isomorphic to the tensor product $\Mn\otimes N$, and if $\tau$ is a normal and faithful trace on $N$ then $\tr_n\otimes \tau$ is a normal and faithful trace on $\Mn(N)$. Moreover $\Mn(N)$ is a factor if and only if $N$ is a factor.

Projections in tracial von Neumann algebras have a rich structure. If $p$ is a non-zero projection in a von Neumann algebra $M$ then $pMp$ is also a von Neumann algebra, often referred to as a corner of $M$. $pMp$ is a factor provided that $M$  is a factor \cite[\S I.2]{dixmier2011neumann}.   If $\tau$ is a (normal) trace on $M$ then $\frac{1}{\tau(p)}\tau\restrict{pMp}$ is a (normal) trace on $pMp$.   Suppose $(M,\tau)$ is a finite factor.  If $p_1,p_2$ are projections with $\tau(p_1)=\tau(p_2)$, then there exists a partial isometry $v\in M$ for which $v^*v=p_1$ and $vv^*=p_2$. If $M$ is moreover infinite-dimensional, then for any two projections $p_1\leq p_2$ and for any $r\in [0,1]$ with $\tau(p_1)\leq r \leq \tau(p_2)$ there exists a projection $p_1\leq q\leq p_2$ with $\tau(q)=r$. These facts may be found in  \cite{anantharaman2017introduction}. 

Let $A$ be a \csalg and $\pi:A\to \mathcal{B}(\mathcal{H})$ a unital $*$-representation. The von Neumann algebra $\pi(A)''\subseteq \BH$ is called \emph{the von Neumann algebra of $\pi$}. 
A \emph{tracial representation} of $A$ is a pair $(\pi,\tau)$ where $\pi$ is a unital $*$-representation and $\tau$ is a normal and faithful trace on $\pi(A)''$. 
Two tracial representations $(\pi_1,\tau_1)$ and $(\pi_2,\tau_2)$ are said to be \emph{quasi-equivalent} if there exists a $*$-isomorphism between the von Neumann algebras, $\Phi:\pi_1(A)''\to\pi_2(A)''$  such that $\tau_2 \circ\Phi=\tau_1$. We recall the standard fact that traces parameterize tracial representations as follows:
\begin{Thm}
[See \cite{dixmier1977c,bekka_harp}]
\label{thm:extremal-factor}
    Let $A$ be a separable unital  \csalg. Then there exists a one-to-one correspondence between traces on $A$ and quasi-equivalence classes of tracial representations of $A$. The trace corresponding to a tracial representation $(\pi,\tau)$ is $\varphi=\tau\circ \pi$. Moreover, $\varphi$ is extreme if and only if $\pi(A)''$ is a factor.  
\end{Thm}
A trace on $A$ is said to be \emph{finite-dimensional} if the von Neumann algebra corresponding to it is finite-dimensional.

\subsection{Approximation of traces}
The following two simple lemmas will be used often.  The first one shows that in order to ensure that a trace simplex is Poulsen, it suffices to approximate traces of the form $\frac{1}{2}(\varphi_1 +\varphi_2)$ where $\varphi_1, \varphi_2$ are extreme.

\begin{Lemma}\label{lem:approximating-sums-of-two-extreme-points}

Let $C$ be a metrizable compact convex set. Suppose that for any two extreme points $x_1, x_2 \in \partial_e C$, the mid-point $\frac{1}{2}(x_1+x_2)$ lies in the closure of $\partial_e C$. Then $\partial_e C$ is dense in $C$.
\end{Lemma}

\begin{proof}
By induction, any finite convex combination of extreme points with coefficients in $\mathbb{Z}\left[\frac{1}{2}\right]$ -- and hence with arbitrary coefficients -- can be approximated by an extreme point. Thus by Krein-Milman, any point in $C$ can be approximated by a convex combination of extreme points.
\end{proof}

The second lemma shows that in order to approximate a trace, it is enough to approximate the corresponding tracial representations on a generating set, in terms of the trace norm.

\begin{Lemma}\label{lem:approximate in norm on generators}
Let $A$ be a unital \csalg generated by a set $S$. Let $\pi$ be a $*$-representations of $A$ into a tracial von Neumann algebra $(M,\tau)$. Assume that for any $\eps>0$ there exists a $*$-representation $\pi_\eps:A\to M$ such that 
\begin{equation}\label{eq:approximate in norm on generators}
    \forall x\in S: \quad \|\pi_\eps(x) - \pi(x)\|_\tau <\eps  
\end{equation}
Then we have the following convergence of traces:
\[
    \forall x\in A: \quad \lim_{\eps\to 0} \frac{1}{\tau (\pi_\eps(1))}\tau \circ \pi_\eps(x)=
    \frac{1}{\tau(\pi(1))}\tau \circ \pi(x)
\]  
\end{Lemma}

We remark that the $*$-representations in this lemma are not assumed to be unital.

\begin{proof}
The limit $\lim_{\eps\to 0}\|\pi_\eps(x)-\pi(x)\|_\tau =0$ holds true for $x\in S$, and thus for all $x$  in the dense $*$-algebra $A_0$ generated by $S$. $\tau$ is normal and in particular continuous in the trace norm. The convergence  $\tau\circ \pi_\eps(x)\to \tau\circ \pi (x)$ therefore holds for any $x\in A_0$. The weak-* compactness of the unit ball in $A^*$ implies that $\tau\circ \pi_\eps(x)\to \tau\circ \pi(x)$ holds true for all $x\in A$ and in particular for $1\in A$. The statement thus follows. 
\end{proof}

\section{Free products of commutative \csalgs} \label{sec:free group}
This section is devoted to proving Theorem \ref{Thm:free products of commutative algebras}. We start by presenting a general method for perturbing representations. 
\subsection{Perturbations of representations}
We recall that any two uncountable Polish spaces are Borel isomorphic
     \cite[Thm 3.3.13]{srivastava2008course}. The following lemma is a refinement of this fact.
\begin{Lemma}\label{lemma:Polish-magic}
Let $(X,d)$ be a compact metric space with no isolated points. Let $Z$ be a Polish space and $p:Z\to X$ a continuous  surjective map. 
Then for any $\eps>0$, there exists a Borel isomorphism $\tilde{p}: Z \to X$ such that $d(\tilde{p}(z), p(z))<\eps$ for all $z \in Z$.
\end{Lemma}
\begin{proof}
    Fix $\eps>0$. The assumptions on $X$  guarantee the existence of  a finite partition 
    \[
        X=X_1\cup...\cup X_N, \qquad \text{for some }N\in \N,
    \] 
    where each $X_j$ is an uncountable $G_\delta$-set  of diameter at most $\eps$. Indeed, here is one possible way to construct such a partition. Choose a maximal  $\eps$-separated set $F=\{x_1,...,x_N\}\subseteq X$.  For each $j=1,...,N$ consider the corresponding closed Voronoi cell
     \[
    C_j= \{ x \in X \: |\;\forall k: d(x, x_j) \leq d(x, x_k) \:  \}.
    \]
    Then, to ensure that the sets are disjoint, define
    \[
    X_j= C_j \setminus \bigcup_{i<j} C_i.
    \]
    The property of the set $F$ ensures that $X_1,...,X_N$ form a partition of $X$, and that each of them has diameter at most $\eps$. Moreover, since $X$ has no isolated points then any ball inside it is uncountable (because any complete metric space without isolated points is uncountable). In particular, each $X_j$ is uncountable. 

    Having fixed such a partition of $X$ with the aforementioned properties,  we proceed by noting that the pre-images  $Z_j=p^{-1}(X_j)$ also form  a partition of $Z$ into finitely many uncountable $G_\delta$-sets. 
    Recall that $G_\delta$-sets in a Polish space are Polish as well. 
    We may therefore fix for each $j$ some Borel isomorphism $\tilde{p}_j: Z_j \to X_j$.
    Then define $\tilde{p}$ by $\tilde{p}(z) = \tilde{p}_j(z)$ for the unique $j$ for which $z\in Z_j$. 
    Clearly, $\tilde{p}$ is a Borel isomorphism from $Z$ to $X$. The remaining requirement is satisfied, since for any $z \in Z$, both $p(z)$ and $\tilde{p}(z)$ belong to the same $X_j$ which has diameter at most $\eps$.
\end{proof}

The lemma presented below is crucial for the proof of Theorem \ref{Thm:free products of commutative algebras}.
\begin{Lemma}\label{lemma:unitary_packing} 
    Let $X$ be a compact metrizable space with no isolated points, and consider a non-degenerate unital $*$-representation  $\pi: C(X) \to \BH$ where $\Hil$ is a separable Hilbert space. Suppose that $M$ is a commutative von Neumann subalgebra of $\BH$ containing $\pi(C(X))$. Then $\pi$ can be approximated arbitrarily well in the compact-open topology by a unital $*$-representation $\tilde\pi: C(X) \rightarrow \BH$ such that $W^*(\tilde{\pi}(C(X)))= M$.
\end{Lemma} 
To clarify, the lemma states that for any $\eps>0$ and compact set $K\subseteq C(X)$ there exists $\tilde{\pi}:C(X)\to M$ such that $\sup_{g\in K}\|\pi(g)-\tilde{\pi}(g)\|<\eps$, and at the same time $\tilde{\pi}(C(X))$ is $WOT$-dense in $M$. 
\begin{proof}
$M$ is a von Neumann algebra on a separable Hilbert space and as such, it admits a WOT-dense norm-separable $C^*$-subalgebra $A$. Since $\pi(C(X))$ is separable as well, we may assume that $A$ contains $\pi(C(X))$.  The inclusion $\pi(C(X))\subset A$ induces a surjective continuous map $p:Z\to X$ where $Z$ is the Gelfand spectrum of $A$. By \cite[Thm II.2.5]{davidson1996c}, there is a regular Borel measure $\mu$ on $Z$ and an isomorphism $\rho \fromto{L^\infty(Z,\mu)}{M}$. We note that $\pi(g)=\rho(g\circ p)$ for all $g\in C(X)$.

Fix a metric $d$ on $X$ compatible with its topology. Let $\eps>0$ and $K\subset C(X)$ a compact subset. By Arzela-Ascoli, $K$ is equicontinuous, namely, there exists $\delta>0$ such that if $d(x_1,x_2)<\delta$ then $d(g(x_1),g(x_2))<\eps$ for all $g\in K$ and $x_1,x_2\in X$.
We invoke Lemma \ref{lemma:Polish-magic} with $\delta$ in place of $\eps$ to obtain a Borel isomorphism $\tilde{p}\fromto{Z}{X}$ satisfying $d(\tilde{p}(z),p(z))<\delta$ for all $z\in Z$.
We define $\tilde \pi\fromto{C(X)}{M}$ by $\tilde \pi(g) = \rho(g\circ \tilde{p})$.

First we show that $W^*(\tilde\pi(C(X)))=M$. 
Letting $\nu = \tilde{p}_*(\mu)$ be the push-forward measure, $\tilde{p}$ becomes an isomorphism of measure spaces, i.e. a measure-preserving Borel isomorphism.
Thus, the map $\psi(g) := g\circ \tilde{p}$ defines an isomorphism of von Neumann algebras between $L^\infty(X, \nu)$ and  $L^\infty(Z, \mu)$ which is a homeomorphism w.r.t. the weak-$*$ topologies coming from the corresponding $L^1$-spaces.
Since $C(X)$ is weak-$*$ dense in $L^\infty(X,\nu)$, we have that $\psi(C(X))$ is weak-$*$ dense in $L^\infty(Z, \mu)$. Since $\rho$ is also a homeomorphism from the weak-$*$ topology on $L^\infty(Z,\mu)$ to the weak operator topology on $M$ (as provided by \cite[Thm II.2.5]{davidson1996c}), then $\tilde{\pi}(C(X))$ is WOT-dense in $M$ as claimed.

It remains to show that $\tilde\pi$ approximates $\pi$ up to $\eps$ on $K$. 
By the choice of $\delta$, we have that $\|g\circ p - g\circ \tilde{p}\|<\eps$ for all $g\in K$. Thus,
\begin{align*}
    \|\pi(g) - \tilde\pi(g)\| = \|\rho(g\circ p) - \rho(g\circ \tilde{p})\|
     = \|g\circ p - g\circ \tilde{p}\|
     < \eps.
\end{align*}
\end{proof}

\subsection{Generators for $\Mn$}\label{subsec:generators-of-Mn}
 We will need the following basic lemma regarding $\Mn$ in the proof. Let $I_n$ denote the identity $n$-by-$n$ matrix, $E_{ij}$ (for $1\leq i,j \leq n$) the $n$-by-$n$ matrix with $1$ at its $(i,j)$-entry and $0$ elsewhere, and $C_n$ the permutation matrix representing the cycle $(1 2 \dots n)$, that is\[ C_n =
  \begin{pmatrix}
0 & 0 & \hdots & 1 \\
1 & 0 & \hdots & 0\\

        \vdots &\ddots  &\ddots& \vdots \\
        0 & \hdots & 1 & 0 \rule[-1ex]{0pt}{2ex}
  \end{pmatrix}
\]
Define the following block matrices in $\Mn$: 
\[
U_{k,n} = \left(\begin{array}{@{}c|c@{}}
I_k
  & 0\\
\hline
  0 &
  C_{n-k}
\end{array}\right) ,  \; \; \;
V_{k,n} = 
\left(\begin{array}{@{}c|c@{}}
C_{k}
  & 0\\
\hline
  0 &
  I_{n-k}
\end{array}\right), \; \; \;
\]

\begin{Lemma}
\label{generating_matrices}
For $n \in \mathbb{N}$ and $0 \leq k < n$, the matrices $U_{k,n},V_{k+1,n},E_{11}$
generate $\Mn$ as an algebra.
\end{Lemma}
\begin{proof}
Recall that multiplying a matrix by a permutation matrix from the left/right results in applying this permutation on its rows/columns, respectively. The permutations represented by  $U_{k,n},V_{k+1,n}$ are the cycles $\left(k+1...n\right)$ and $\left(1...k+1\right)$ and the two together act transitively on the set $\{1,..,n\}$. This means that the orbit of $E_{11}$ contains the element $E_{ij}$ for every $1 \leq i,j \leq n$, which generate $\Mn$.
\end{proof}

\subsection{Proof of Theorem \ref{Thm:free products of commutative algebras}}
Consider two compact metrizable spaces $Y$ and $Z$ with no isolated points. Then the corresponding \csalgs  $B=C(Y)$ and $C=C(Z)$ are unital and separable. Consider the free product \csalg $A=B*C$. It will be clear that the proof below works just as well for $A=B*C*D$ where $D$ is any unital separable \csalg, but for simplicity of notation, we ignore this third free factor. We shall show that the trace space of $A$ is a Poulsen simplex. By Lemma \ref{lem:approximating-sums-of-two-extreme-points} it is enough to approximate an element of the form $\varphi=\frac{1}{2}(\varphi_1+\varphi_2)$ where $\varphi_1$ and $\varphi_2$ are extreme traces. 

Fix a finite subset $F$ in $A$ and $0<\epsilon<1$.  Our goal is to find an extreme trace $\varphi '$ of $A$ such that $|\varphi(a)-\varphi '(a)|<\epsilon$ for all   $a \in F$. 

We shall assume without loss of generality that $F$ is contained in the algebraic free product, that is in the dense $*$-algebra generated by the components $B\cup C\subset A$. In particular, there are finite sets $F_B\subset B $  and $F_C\subset C$, and a positive integer $d$ such that
each of the  element of $F$ is a linear combination of at most $d$ summands, and  each such summand  is a  product of at most $d$ elements of $F_B\cup F_C$.  Upon normalizing, we may further assume that the coefficients of each of those linear combinations are of modulus at most $1$, and that $F_B$ and $F_C$ are contained the unit ball. 
\\

Let us briefly explain our strategy. In the first step, we construct a representation $\pi$ into a von Neumann algebra $M$ with trace $\tau$ such that $\tau\circ \pi=\varphi$. While $M$ is a factor, $\pi(A)''\subseteq M$ is not. In the second step, we define a representation $\pi_0$ which is an asymmetric modification of $\pi$. This symmetry-breaking is crucial for later enriching the algebra $\pi_0 (A)''$ via small perturbations. While $\tau\circ \pi_0$ may not agree with $\tau\circ \pi$, it can be chosen to be as close to it as we wish in trace. In the third step we apply Lemma \ref{lemma:unitary_packing} in order to perturb $\pi_0$ into a new representation $\tilde{\pi}$. In the remaining steps we show that $\tilde{\pi}(A)''=M$ and that $\tilde{\pi}(a)$ and $\pi(a)$ are close in norm for every $a\in F$. This will imply that $\tilde{\varphi}=\tau \circ \tilde{\pi}$ is an extreme trace which is close to $\varphi$. Let us now elaborate.\\

{\bf Step 1: defining a joint representation.}
For $i=1,2$, choose tracial representations $(\rho_i,\tau_i)$ so that $\varphi_i=\tau_i\circ \rho_i$. Since $\varphi_i$ is an extreme trace, $N_i:=\rho_i(A)''$ is a factor. We note that as $B$ and $C$ are separable the representations $\rho_1$ and $\rho_2$ are defined on a separable Hilbert space.  

One could associate with $\varphi$ a tracial representation as a direct sum of $\rho_1$ and $\rho_2$.
However, we shall choose a tracial representation into a larger von Neumann algebra so that there is room for perturbations in it.
To that end, fix an even integer $n>\frac{4}{\eps}$, and  consider the algebra $M:=\Mn(N_1 \otimes N_2)$ of $n$-by-$n$ matrices with values in $N_1\otimes N_2$. $M$ is isomorphic to $\Mn\otimes N_1 \otimes N_2$, which is a tensor product of factors and thus a factor itself.
The unique trace on it is $\tau:x\mapsto \frac{1}{n} \sum_{i=1}^{n} \tau_1 \otimes \tau_2 (x_{ii})$. 

Define the following representation $\pi:A\to M$ by:

\begin{equation*}
a \mapsto 
 \left(
    \begin{array}{r@{}c|c@{}l}
  &    \begin{smallmatrix}
        \rho_1(a) \otimes 1 & & 0 \\
          &\ddots&\\
        0 & & \rho_1(a) \otimes 1 \rule[-1ex]{0pt}{2ex}
      \end{smallmatrix} & \raisebox{-5pt}{\mbox{\huge0}}& \rlap{\kern6mm$\Big\updownarrow \frac{1}{2} n $}\\\hline
  &    \raisebox{-5pt}{\mbox{\huge0}} &  
       \begin{smallmatrix}\rule{0pt}{2ex}
         1 \otimes \rho_2(a) & &  0  \\
        
         & \ddots & \\
        0  & & 1 \otimes \rho_2(a)
      \end{smallmatrix}    &  \rlap{\kern6mm $\Big\updownarrow \frac{1}{2} n$}
    \end{array} 
\right)
\end{equation*}
Note that $\varphi = \tau \circ \pi$. \\

{\bf Step 2: desymmetrizing the representation}
Consider the representation $\pi_0:A\to M$ which agrees with $\pi$ on $B$, but on $C$ we slightly change the block structure:
\begin{equation*}
C \ni c
 \mapsto 
 \left(
    \begin{array}{r@{}c|c@{}l}
  &    \begin{smallmatrix}
        \rho_1(c) \otimes 1 & & 0 \\
          &\ddots&\\
        0 & & \rho_1(c) \otimes 1 \rule[-1ex]{0pt}{2ex}
      \end{smallmatrix} & \raisebox{-5pt}{\mbox{\huge0}}& \rlap{\kern6mm$\Big\updownarrow \frac{1}{2} n+1 $}\\\hline
  &    \raisebox{-5pt}{\mbox{\huge0}} &  
       \begin{smallmatrix}\rule{0pt}{2ex}
         1 \otimes \rho_2(c) & &  0  \\
        
         & \ddots & \\
        0  & & 1 \otimes \rho_2(c)
      \end{smallmatrix}    &  \rlap{\kern6mm $\Big\updownarrow \frac{1}{2}n-1$}
    \end{array} 
\right)
\end{equation*}

In particular, for any $c\in C$, $\pi(c)$ and $\pi_0(c)$ agree at all entries except for two. The trace $\varphi_0:=\tau\circ \pi_0$ satisfies that for all $a$ in the unit ball of $A$ (and in particular for any $a\in F$):

\begin{equation} \label{first approx}
    |\varphi(a)-\varphi_0(a)|=|\tau(\pi(a)-\pi_0(a))|\leq\frac{2}{n}<\frac{\eps}{2}.
\end{equation}\\

{\bf Step 3: perturbing the representation}
While $M$ is a factor, the von Neumann subalgebra generated by the image of $\pi_0$ is not. The next step is, therefore, to use Lemma \ref{lemma:unitary_packing} in order to perturb the representation slightly in a way that its image will generate $M$, and thus will be a factor representation. 
To that end, set $U:= U_{\frac{1}{2} n,n}$, $V:= V_{\frac{1}{2} n+1,n}$, $P:=E_{11}$ as defined in \S\ref{subsec:generators-of-Mn}, and think of them as elements of $M$ via the natural embedding $\Mn(\mathbb{C})\subseteq M$. The von Neumann algebra  $W^*(\pi_0(B),U,P)$ is commutative. We may therefore apply  Lemma \ref{lemma:unitary_packing} to the restricted representation  ${\pi_0}_{\mkern 1mu \vrule height 2ex\mkern2mu B}$, and as a result
obtain a new representation $\sigma_B: B \rightarrow M$ such that $W^*(\sigma_B(B)) = W^*(\pi_0(B), U,P)$ and moreover $\max_{b\in F_B}\| \sigma_B(b)-\pi_0(b)  \| < \frac{\eps}{2d^2}$.
We apply Lemma \ref{lemma:unitary_packing} once again, this time for the representation ${\pi_0}_{\mkern 1mu \vrule height 2ex\mkern2mu C}$ and the commutative von Neumann algebra $W^*(\pi_0(C),V)$. We obtain a representation $\sigma_C: C \rightarrow M$ such that  $W^*(\sigma_C(C)) = W^*(\pi_0(C), V)$ and $\max_{c\in F_C}\| \sigma_C(c) - \pi_0(c) \| < \frac{\eps}{2d^2}$. We set $\tilde{\pi}:=\sigma_B*\sigma_C$, that is, the unique representation of $A$ whose restriction to $B$ and $C$ is $\sigma_B$ and $\sigma_C$, respectively. Its corresponding trace is denoted by  $\tilde{\varphi}=\tau\circ \tilde{\pi}$.
\\

We have thus defined our final perturbed trace $\tilde{\varphi}$. The proof will be complete once we prove that $\tilde{\varphi}$ is extreme, and that it approximates $\varphi$.

\begin{Claim} $\tilde{\varphi}$ is extreme.
\end{Claim}
\begin{proof}
This is equivalent to showing that $\tilde{\pi}$ is a factor representation. Denote 
$M_0:=\tilde{\pi}(A)'' \subseteq M\cong \Mn\otimes N_1\otimes N_2$. 
Then  $M_0$ contains $\sigma_B(B)$ and $\sigma_C(C)$, and therefore it contains also $W^*(\pi_0(A), U,V,P)$. By Lemma \ref{generating_matrices}, $U$, $V$ and $P$ together generate $\Mn$, so in particular, each $E_{ij}$ is in $M_0$ as well as every permutation matrix. We conclude that $\rho_1(A)\otimes 1 \otimes E_{ij}$ and  $1 \otimes \rho_2(A) \otimes E_{ij}$ belong to $M_0$ for any $1\leq i,j\leq n$. Indeed, by multiplying $\pi_0(a)$ and $\pi_0(a)$ by $E_{11}$ or by $E_{nn}$ we see that the above operators are in $M_0$ for particular $i,j$, and by further multiplying them by permutation matrices, we see that all of the above are in $M_0$. Upon taking the WOT-closures we get that  $N_1$ and $N_2$ are contained in $M_0$. Hence $M_0=M$.
\end{proof}
 
\begin{Claim}
$|\tilde{\varphi}(a)-\varphi(a)|<\eps$ for any $a\in F$.
\end{Claim}
\begin{proof}
We have that $\|\pi_0(a)-\tilde{\pi}(a)\|<\frac{\eps}{2d^2}$ for all $ a \in F_B\cup F_C$. We recall that each element of $F$ is a linear combination of $d$ summands with coefficients of modulus at most $1$, and that each of those summands is a product of at most $d$ elements of $F_B\cup F_C$. Thus $\|\pi_0(a)-\tilde{\pi}(a)\|<\frac{\eps}{2}$ holds for all $a\in F$. As a result
\begin{equation}\label{second approx}
    | \tilde{\varphi}(a)-\varphi_0(a)|=| \tau(\tilde{\pi}(a)-\pi_0(a))|\leq \|\pi_0(a)-\tilde{\pi}(a)\|<\frac{\eps}{2}, \qquad \forall a\in F. 
\end{equation}
Combining (\ref{second approx}) with (\ref{first approx}) we get:
\begin{equation}
    \sup_{a\in F}| \tilde{\varphi}(a)-\varphi(a)|<\eps 
\end{equation}
The proof is complete.
\end{proof}

\section{Free products of matrix algebra} \label{sec:matrix algebras}
In this section, we focus our attention on traces on the free product algebra $\MnMn$ for $n\geq 4$ and prove  Theorem \ref{Second Theorem}. 

\subsection{Generating sets in von Neumann algebras}

We will often need to argue that certain sets of operators generate a given von Neumann algebra.

\begin{Lemma}\label{lem:generating-factors-from-corners}
	Let $M$ be a von Neumann algebra,  and let $q_1,q_2\in M$ be projections with $q_1+q_2=1$.  Assume $v \in M$ is a partial isometry with $v^*v=q_1$ and such that $vv^*$ is a sub-projection of $q_2$. Then the corner $q_2Mq_2$ together with $v$ generate $M$ as a von Neumann algebra.
\end{Lemma}

\begin{proof}
    Let $S$ be the $*$-algebra generated by $\{v\}\cup q_2Mq_2$. Then $S$ contains, in particular, the elements $q_1,q_2,v,v^*$. We will show that $M \subset S''$. 
    Consider arbitrary $x\in M$ and $y\in S'$. Then a direct computation shows that $q_ixyq_j=q_iyxq_j$ for $i,j\in \{0,1\}$. 
    For example, the fact that $v^*$, $q_1$ and $q_2vxq_2$ are in $S$ implies that
    \[
        q_1xyq_2=q_1xq_2y=v^*q_2vxq_2y=yv^*q_2vxq_2=yq_1xq_2=q_1yxq_2 
    \]
    The other cases are obtained in a similar way. 
    Hence: 
    \begin{align*}
        xy=(q_1+q_2)xy(q_1+q_2)=
        (q_1+q_2)yx(q_1+q_2)=yx
    \end{align*}
\end{proof}

\begin{Lemma}\label{lem:generate tensors}
Let $M_1, M_2$ be \vnalgs, with $p \in M_1$ a non-zero projection, and assume $M_1$ is a factor. Then the algebra $p \otimes  M_2 $, together with $M_1 \otimes 1_{M_2}$ generate $M_1 \otimes M_2$.
\end{Lemma}

\begin{proof}
Since $M_1\otimes M_2$ is generated by $M_1\otimes 1_{M_2}$ and $1_{M_1}\otimes M_2$, it suffices to show that we can generate $1_{M_1}\otimes M_2$.

Since $M_1$ is a factor, it has no nontrivial WOT-closed ideals, so $1_{M_1}$ can be approximated (in the WOT) by elements of the form $\sum_{i=1}^n x_ipy_i$ where $x_i, y_i\in M$. Therefore, for any $z\in M_2$, we can approximate $1\otimes z$ by elements of the form
$$\left(\sum_{i=1}^n x_{i}^*py_i\right) \otimes z = \sum_{i=1}^n (x_i\otimes 1_{M_2})(p\otimes z)(y_i\otimes 1_{M_2}).$$
\end{proof}

\subsection{An amplification lemma}
We will need the following standard construction, which allows one to embed a given factor as a corner in a slightly larger one:

\begin{Lemma}\label{lem:embedding N as a corner}
Let $(N,\tau)$ be a finite factor and let $r\in N$ be a projection. Then there exists a finite factor $(\tilde{N},\tilde{\tau})$ and a projection $q\in \tilde{N}$ such that $N\cong q\tilde{N}q$ and $\tilde{\tau}(1_{\tilde{N}}-q)=\tilde{\tau}(r)$.
Moreover, the finite factor $\tilde{N}$ may be chosen to be a corner of a tracial von Neumann algebra $M$ which does not depend on $r$.
\end{Lemma}

\begin{proof}
Embed $N$ as the upper left corner of $M=\Matrix_2(N)$, and define $q=E_{11}^{(2)}\otimes 1_{N}\in M$.  Let $q':=q+ E_{22}^{(2)}\otimes r$ and set $\tilde{N} = q'Mq'$. Then $\tilde{N}$ is a factor since it is a corner of a factor, and let $\tilde{\tau}$ be the unique trace on it. The fact that $N\cong q\tilde{N}q$ is clear from the definitions. Moreover, $1_{\tilde{N}} - q$ is equivalent to $r$ (as witnessed by the partial isometry $E_{12}^{(2)}\otimes r$). It follows that $\tilde{\tau}(1_{\tilde{N}}-q)=\tilde{\tau}(r)$.  
\end{proof}
\begin{Rmk} 
The construction of $\tilde{N}$ above is called the amplification of $N$ by $t=1+\tau(r)$ often denoted by $N^t$. See \cite[\S4.2]{anantharaman2017introduction}.
\end{Rmk}

\subsection{Representations of $\Mn$}
The following facts are standard and may be found in greater detail in \cite{kadison1986fundamentals}.
Let  $\pi: \Mn  \to \BH$ be a unital $*$-representation. Then $\pi(E_{11}),...,\pi(E_{nn})$ are mutually orthogonal projections that sum up to $1_\Hil$. These projections therefore define a direct sum decomposition $\Hil=\Hil_1 \oplus...\oplus \Hil_n$, and each $\pi(E_{ij})$ for $i\ne j$ must be a  partial isometry from $\Hil_j$ onto $\Hil_i$. 
Conversely, if $\Hil$ admits a direct sum decomposition $\Hil=\Hil_1\oplus...\oplus \Hil_n$ with corresponding projections $p_1,...,p_n$, and $v_2,\ldots, v_n$ are partial isometries from $\Hil_j$ to $\Hil_1$, then the mapping $\pi$ that takes $E_{ii}\mapsto p_i$ and $E_{ij}\mapsto v_{i}^*v_{j}$ (where $v_1 = p_1$) is a unital $*$-representation of $\Mn$ into $\BH$. 

Furthermore, assume that $\pi$ is a unital representation of $\Mn$ into some \vnalg $M$. Then there exists a von Neumann algebra $N$ and an isomorphism $M\cong \Mn(N)$ such that under this isomorphism, $\pi$ is simply the standard representation $\Mn\to \Mn(\C)\subseteq \Mn(N)$  \cite{kadison1986fundamentals}. $N$ is in fact taken to be the commutant of $\pi(\Mn)$.  Clearly, $M$ is tracial (resp. a factor) if and only if $N$ is. We add that the same conclusion is true for \csalgs by the same proof, or by \cite[Thm I.10.7]{davidson1996c}.

Finally, we note that $\Mn$ is generated by the set $\{E_{1j}\}_{j=2}^n$ and thus any unital $*$-representation is determined by its values on this set. Putting these observations altogether we obtain:

\begin{Lemma}\label{lem:reps of Mn, genearotrs and relations}
Let $N$ be a von Neumann algebra, and suppose $u_2,...,u_n \in N$ are unitaries. Then the assignment:
\[
    E_{1j}\mapsto u_j\otimes E_{1j} \in \Mn(N) 
\]
extends in a unique way to a unital *-representation $\Mn\to \Mn(N)$.
\end{Lemma}

\subsection{Representations of $\MnMn$}\label{reps of Mn and MnMn}
Let us now turn our attention to representations of $\MnMn$.
This algebra is generated by two sets of matrix units, which we denote $e_{ij}$ and $f_{ij}$ (we continue to write $E_{ij}$ for the matrix units in $\Mn$ --- the new notation refers to specific elements in the algebra $\MnMn$).
In the proof of Proposition 3.7 (ii) of \cite{musat2021factorizable}, the authors construct unital homomorphisms with domain $\MnMn$ as follows:
let $A$ be a unital \csalg and let $u_2, \ldots, u_n \in A$ be unitaries.
Define a homomorphism $\pi:\MnMn\to \Mn(A)$ by $\pi(e_{ij}) = 1_A \otimes E_{ij}$ for $i,j=1,\ldots,n$ and $\pi(f_{1j}) = u_j\otimes E_{1j}$ for $j=2,\ldots,n$.
We note that $\pi$ is surjective iff $u_2,\ldots,u_n$ generate $A$.
We further note that $\pi$ satisfies $\pi(e_{ii}) = \pi(f_{ii})$ for $i=1,\ldots,n$. 

Conversely, let $\pi$ be any unital homomorphism with domain $\MnMn$ satisfying $\pi(e_{ii}) = \pi(f_{ii})$ for $i=1,\ldots,n$.
Then by the discussion preceding Lemma \ref{lem:reps of Mn, genearotrs and relations}, we may identify the image of $\pi$ with $\Mn(A)$ for some unital \csalg $A$, in such a way that $\pi(e_{ij}) = 1_A\otimes E_{ij}$.
By assumption, we also have $\pi(f_{ii}) = 1_A\otimes E_{ii}$, which implies that $\pi(f_{1j}) = u_j\otimes E_{1j}, j=2,\ldots,n$ for some unitaries $u_2,\ldots,u_n\in A$ which generate $A$.
Letting $D$ denote the subalgebra of $\Mn$ consisting of diagonal matrices, we have described how the construction of R\o rdam and Musat leads to the following:

\begin{Prop}[M. R\o rdam and M. Musat] \label{prop:amalgamated}
$\Mn *_D \Mn \cong \Mn(C^*(F_{n-1}))$.
\end{Prop}
As $\Mn$ admits a unique trace, we get an isomorphism between the simplex of traces of the free group $F_{n-1}$ and the simplex of traces of $\Mn *_D \Mn$.
Combining this with Corollary \ref{Cor Free Group}, we deduce the following:

\begin{Cor}\label{cor:amalgamated poulsen}
The trace space of $\Mn *_D \Mn$  is a Poulsen simplex, for $n\geq 3$.
\end{Cor}
The face of $\Traces{\MnMn}$ consisting of traces which factorize through $\Mn *_D \Mn$ is thus a Poulsen simplex. In the proof of Theorem \ref{Second Theorem} (as in the proof of Theorem \ref{Thm:free products of commutative algebras}), we will define a representation into a suitably constructed factor, and most of the work will consist of making ``tracially small" modifications in order to generate the factor. Our strategy will be to arrange for our representation to have a ``tracially small" sub-representation which factors through $\Mn *_D \Mn$, which we will use in the spirit of Proposition \ref{prop:amalgamated}. 

\subsection{Proof of Theorem \ref{Second Theorem}} \label{Mn main proof section}

We now proceed with the proof of the Theorem. We will sometimes abuse notation and ignore the ordering of tensor components (i.e. identify $x\otimes y$ with  $y\otimes x$). This will not cause any ambiguity, as we shall do so only when the tensor products  are composed of distinct algebras. We will also introduce a superscript to matrix units, e.g $E_{ij}^{(n)}$, to indicate the dimensions of the matrix.

 Denote $A=\MnMn$.
By Lemma \ref{lem:approximating-sums-of-two-extreme-points} it is enough to approximate traces of the form $$\varphi=\frac{1}{2}(\varphi_1+\varphi_2)$$ where $\varphi_1$ and $\varphi_2$ are extreme traces.

Choose tracial representations $\pi_i$ of $A$ corresponding to the traces $\varphi_i$ for $i=1,2$. According to the discussion preceding Lemma \ref{lem:reps of Mn, genearotrs and relations}, we may assume that the von Neumann algebras of the representations $\pi_i$ are of the form $\Mn\otimes N_i$ for some finite factors $N_i$, such that $\pi_i(e_{jk}) = E_{jk}^{(n)}\otimes 1_{N_i}$. Let $M=\Mn\otimes N_1\otimes N_2\otimes \Matrix_2$ and denote the tensor product trace by $\tau$. Then $\varphi=\tau\circ \pi$, where $\pi$ is the unital  $*$-representation $\pi:A\to M$ defined by:
\begin{equation}
    \pi(x)=\pi_1(x)\otimes 1_{N_2}\otimes E_{11}^{(2)}
          +\pi_2(x)\otimes 1_{N_1}\otimes E_{22}^{(2)}      
\end{equation}
Fix $\eps>0$. In \S\ref{section:findim}, we will introduce a dimension-amplification method for finite-dimensional traces. We show, in particular, that (extreme) finite-dimensional traces can be approximated by finite-dimensional extreme traces with arbitrarily high dimension, see Corollary \ref{cor: finite-dimensional reps of Mn are poulsen}.  It is therefore harmless to assume that the dimension of $N_1$ and $N_2$ is at least $1/{\eps^2}$ (or infinite-dimensional). Moreover, if both $N_1$ and $N_2$ are finite-dimensional, one can assume that they are of the same dimension, see Remark \ref{rmk:approximating a single extremal finite-dimensional}.

Let $ S = \{ e_{12},...,e_{1n}, f_{12},...,f_{1n} \}$  denote the set of generators of $A$. We shall now construct a finite factor $(\tilde{M},\tilde{\tau})$, and a unital *-representation $\tilde{\pi}$ of $A$ into $\tilde{M}$ such that the following holds:
\begin{enumerate}
    \item \label{three-conditions: M} There exists a finite factor $(\hat{M},\hat{\tau})$ independent of $\eps$ (and of $\tilde{M}$) which contains $\tilde{M}$ and $M$ as corners.  
    \item \label{three-conditions: approx} $\|\tilde{\pi}(x)-\pi(x)\|_{\hat{\tau}}\leq 4\eps$ for $x\in S$.
    \item \label{three-conditions: generating} $\tilde{\pi}(A)''=\tilde{M}$.
\end{enumerate}
Once this is done, the proof of the theorem is complete. Indeed,   (\ref{three-conditions: generating}) ensures that the trace $\tilde{\varphi} = \tilde{\tau} \circ \tilde{\pi}$ is extreme. By (\ref{three-conditions: M}) the representations $\tilde{\pi}$ and $\pi$ may be viewed as (not necessarily unital) $*$-representations into $\hat{M}$. Together with (\ref{three-conditions: approx}),
it follows from Lemma \ref{lem:approximate in norm on generators} that $\tilde{\varphi}$ approximates the trace $\varphi$ arbitrarily well,  as $\eps\to 0$.\\

We now turn to construct $\tilde{M}$ and $\tilde{\pi}$ that satisfy the above requirements.
First,  due to the assumption on the dimension of $N_i$, ($i=1,2)$, we can choose projections $r_i \in N_i$ with trace $\tau_{N_1}(r_1) = \tau_{N_2}(r_2) < \eps$. 
Apply Lemma \ref{lem:embedding N as a corner} to obtain finite factors $\tilde{N}_i$ containing $N_i$ as a corner, i.e., $N_i = 1_{N_i} \tilde{N}_i 1_{N_i}$.
Moreover, for the projections $p_i = 1_{\tilde{N}_i} - 1_{N_i}$, we have that $\tau_{\tilde{N}_i}(p_i) = \tau_{\tilde{N}_i}(r_i)$. Thus, there exist partial isometries $v_i\in \tilde{N}_i$ such that $v_iv_i^* = r_i$ and $v_i^*v_i = p_i$.
Denote $\tilde{N}=\tilde{N}_1\otimes \tilde{N}_2$. Define the factor $\tilde{M} = \Mn\otimes \tilde{N} \otimes \Matrix_2$ and denote the corresponding trace by $\tilde{\tau}=\tau_{\tilde{M}}$. 
 The finite factor $(\tilde{M}, \tilde{\tau})$ is the one we will show to satisfy the three requirements above.
 Note that $\tilde{M}$ contains $M$ as a corner and, as stated in Lemma \ref{lem:embedding N as a corner}, $\tilde{M}$  is contained in some finite factor $\hat{M}$ independent of $\eps$. Condition (1) is thus verified.\\

It will be convenient to introduce the non-unital embeddings $\sigma_i\fromto{\tilde{N}_i}{\tilde N \otimes \Matrix_2}$:
\begin{align}
\sigma_1 :\tilde{N}_1\to \tilde{N}\otimes \Matrix_2,\qquad x_1\mapsto x_1\otimes 1_{\tilde{N}_2}\otimes E_{11}^{(2)}\\
\sigma_2 :\tilde{N}_2\to \tilde{N}\otimes \Matrix_2,\qquad x_2\mapsto  1_{\tilde{N}_1}\otimes x_2 \otimes E_{22}^{(2)}
\end{align}
We will also make use of the following projection $p\in \tilde{N}$:
\begin{equation}\label{eq:def-of-p}
\begin{split}
    p = 1_{\tilde{N}}- 1_N &= p_1\otimes 1_{\tilde{N}_2}+1_{\tilde{N}_1}\otimes p_2 - p_1\otimes p_2\\
    &= p_1\otimes 1_{N_2}+1_{N_1}\otimes p_2 +p_1\otimes p_2.
\end{split}
\end{equation}

We are now ready to define the required representation.
Let $\lambda_1, \lambda_2$ denote complex numbers of modulus $1$, which we shall choose later. Using Lemma \ref{lem:reps of Mn, genearotrs and relations}, we define $\tilde{\pi}\fromto{A}{\tilde M}$ as follows:
\begin{align}
    \tilde{\pi}(f_{1j}) &= \left[E_{1j}^{(n)} \otimes p_1 + \pi_1(f_{1j})\right] \otimes 1_{\tilde{N}_2}\otimes E_{11}^{(2)} + \left[E_{1j}^{(n)}\otimes p_2 + \pi_2(f_{1j})\right]\otimes 1_{\tilde{N}_1} \otimes E_{22}^{(2)} \label{def:tilde pi 2} \\[2ex]
    \tilde{\pi}(e_{12}) &= E_{12}^{(n)}\otimes \left[ \sigma_1\left(  1_{\tilde{N}_1} - p_1 -r_1+v_1+v_1^* \right) + \sigma_2\left( 1_{\tilde{N}_2} - p_2 - r_2+v_2 + v_2^* \right)\right] \label{eq: e_12} \\[2ex]
    \tilde{\pi}(e_{13}) &= E_{13}^{(n)} \otimes\left[\left( 1_{\tilde{N}}-p \right) \otimes  1_{\Matrix_2} +  p \otimes  \left( E_{12}^{(2)}+  E_{21}^{(2)}\right)\right]\label{eq: e_13}\\[2ex]
    \tilde{\pi}(e_{14}) &= E_{14}^{(n)}\otimes\left[ \sigma_1(1_{N_1} +\lambda_1 p_1) +\sigma_2(1_{N_2} +\lambda_2 p_2)\right]\label{eq: e_14} \\[2ex]
    \tilde{\pi}(e_{1j}) &= E_{1j}^{(n)}\otimes 1_{\tilde{N}}\otimes 1_{M_2} \label{def:tilde pi 1}, \qquad j=5,\ldots,n.
\end{align}
It is easy to see that the operators $\tilde{\pi}(f_{1j})$ defined above for all $j=2,...,n$, are partial isometries possessing common range projection and mutually orthogonal source projections which sum to $1_{\tilde{M}}$, and thus define a unital representation of $\Mn$. 
Similarly, we see that $\tilde{\pi}(e_{1j})$ are of the form $E_{1j}^{(n)}\otimes v_j$ where $v_j\in \tilde{N}$ are unitary, as specified in Lemma \ref{lem:reps of Mn, genearotrs and relations}. Thus the equations above indeed determine a representation $\tilde{\pi}\fromto{A}{\tilde M}$.

Note that for each $x\in S$, the operator $\tilde{\pi}(x)-\pi(x)$ is supported on a projection which has trace at most $2\eps$. Since in addition $\|\tilde{\pi}(x)-\pi(x)\| \le 2$, we have  $$\|\tilde{\pi}(x)-\pi(x)\|_{\hat{\tau}}\leq \|\tilde{\pi}(x)-\pi(x)\|_{\tilde{\tau}}\leq 4\eps$$ 
and condition (2) is thus fulfilled.\\

Let us denote $M_0 = \tilde{\pi}(A)''$. The rest of the proof is devoted to proving that $M_0 = \tilde M$, and thus verifying condition (3). We shall do this in a series of claims. 

\begin{Claim}\label{claim:tilde-p is in M0}

The elements $E_{ij}^{(n)}\otimes \sigma_1(p_1)$ and $E_{ij}^{(n)} \otimes \sigma_2(p_2)$ are in $M_0$ for all $1\leq i,j\leq n$. 
\end{Claim}
\begin{proof}
First, we may rewrite (\ref{def:tilde pi 2}) as 
\begin{equation}
    \tilde \pi (f_{1j}) = E_{1j}^{(n)}\otimes \left [\sigma_1(p_1) + \sigma_2(p_2)\right] + b_j, 
\end{equation}
where $b_j = \pi_1(f_{1j})\otimes 1_{\tilde N_2}\otimes E_{11}^{(2)} + \pi_2(f_{1j})\otimes 1_{\tilde N_1}\otimes E_{22}^{(2)}$. 
Denote $a = e_{14}f_{41}e_{11} \in A$, and consider the element $\tilde\pi(a)\in M_0$:
\begin{align*}
    \tilde\pi(a) & = \tilde\pi(e_{14}) \tilde\pi(f_{14})^*\tilde\pi(e_{11}) \\
    & = \tilde\pi(e_{14})\left(E_{41}^{(n)}\otimes \left [\sigma_1(p_1) + \sigma_2(p_2)\right]\right)\tilde\pi(e_{11}) + \tilde\pi(e_{14})b_4^*\tilde\pi(e_{11})\notag\\
    &= E_{11}^{(n)}\otimes\left[\lambda_1 \sigma_1(p_1) + \lambda_2 \sigma_2(p_2)\right]  + \tilde\pi(e_{14})b_4^*\tilde\pi(e_{11})
\end{align*}
Note that $b_j$ is orthogonal to $1_{\Matrix_n}\otimes\sigma_i(p_i)$ for $i=1,2$. Thus we may develop the second summand above as:
\begin{align*}
    \tilde\pi(e_{14})b_4^*\tilde\pi(e_{11}) & = \left(E_{14}^{(n)}\otimes\left[ \sigma_1(1_{N_1} +\lambda_1 p_1) +\sigma_2(1_{N_2} +\lambda_2 p_2)\right]\right) b_4^* \left(E_{11}^{(n)}\otimes 1_{\tilde N}\otimes 1_{\Matrix_2}\right)\\
    & =\left(E_{14}^{(n)}\otimes 1_{\tilde N} \otimes 1_{\Matrix_2}\right) b_4^* \left(E_{11}^{(n)}\otimes 1_{\tilde N}\otimes 1_{\Matrix_2}\right)\\
    & = \pi_1(a)\otimes 1_{\tilde N_2}\otimes E_{11}^{(2)} + \pi_2(a)\otimes 1_{\tilde N_1}\otimes E_{22}^{(2)}. 
\end{align*}

This expression, together with the fact that $\pi_i(a) \perp p_i$, makes it clear that the eigenvalues of $\tilde{\pi}(a)$ are the union of $\{ \lambda_1, \lambda_2 \}$ and the eigenvalues of $\pi_1(a)$ and $\pi_2(a)$.
Since $a$ is a contraction, by \cite{jamison1965eigenvalues} there are at most countably many eigenvalues of modulus 1 for both $\pi_1(a)$ and $\pi_2(a)$.
We may now, therefore, fix  our choice of $\lambda_1$ and $\lambda_2$, requiring that they are distinct, and are not eigenvalues of $\pi_1(a)$ and $\pi_2(a)$.
By the choice of $\lambda_i$, we see that $E_{11}^{(n)} \otimes \sigma_i(p_i)$ is the projection onto the $\lambda_i$ eigenspace of $\tilde{\pi}(a)$. This projection is in $M_0$ because it can be expressed as the spectral projection onto $\{0\}$ of the normal element $(\tilde{\pi}(a)-\lambda_i)^*(\tilde{\pi}(a)-\lambda_i)$.
Finally, we note that for $k=1,2$:
\[  
    E_{ij}^{(n)}\otimes \sigma_k(p_k) =\tilde{\pi}(f_{i1}) \left ( E_{11}^{(n)} \otimes \sigma_1(p_k) \right) \tilde{\pi}(f_{1j}).
\]
The claim is thus proven.
\end{proof}

\begin{Claim}\label{claim:Mn is in M0}
$\Mn \otimes 1_{\tilde{N}} \otimes 1_{\Matrix_2} \subseteq M_0$.
\end{Claim}

\begin{proof}
Since $\tilde{\pi}(e_{1j})=E_{1j}^{(n)}\otimes 1_{\tilde{N}}\otimes 1_{\Matrix_2}$ for all $j\ne 2,3,4$, we only need to show the claim for $j=2,3,4$.

We start with $j=4$. By the definition of  $\tilde{\pi}(e_{14})$ we see that
\begin{align*}
    E_{14}^{(n)} \otimes  1_{\tilde{N}} \otimes  1_{\Matrix_2} = \tilde{\pi}(e_{14})  + E_{14}^{(n)}\otimes \left[(1 - \lambda_1)\sigma_1(p_1) +(1-\lambda_2) \sigma_2(p_2) \right].
\end{align*}
The second summand is in $M_0$ by the previous claim, therefore $E_{14}^{(n)}\otimes 1_{\tilde{N}}\otimes 1_{\Matrix_2} \in M_0$. 

In order to recover the standard $E_{12}^{(n)}$, we recall that:
\begin{align*}
    \tilde{\pi}(e_{12}) = & E_{12}^{(n)}\otimes 1_{\tilde{N}}\otimes 1_{\Matrix_2}\\ 
    + & E_{12}^{(n)}\otimes \sigma_1\left( v_1 + v_1^* - p_1 - r_1  \right) \\
     + & E_{12}^{(n)}\otimes \sigma_2 \left(v_2 + v_2^*- p_2 - r_2 \right).
\end{align*}
We would  therefore like to isolate elements of the form $E_{12}^{(n)} \otimes \sigma_1(v_1)$, for example,
so that we will be able to subtract them from $\tilde{\pi}(e_{12})$.
To that end, we note that $ v_ip_i = v_i$, and $p_i v_i^* = v_i^*$, thus: 
\begin{equation}\label{eq:v_1}
    E_{12}^{(n)} \otimes\sigma_1(v_1) = 
  \tilde{\pi}(e_{12}) \left ( E_{22}^{(n)}\otimes \sigma_1(p_1) \right )
\end{equation}
and 
\begin{equation}\label{eq:v_1 star}
    E_{12}^{(n)} \otimes \sigma_1(v_1^*) = 
    \left ( E_{11}^{(n)}\otimes \sigma_1(p_1) \right ) \tilde{\pi}(e_{12}).
\end{equation}
Since $E_{ij}^{n} \otimes \sigma_k(p_k)$ are in $M_0$ for every $i,j,k$ due to the previous claim, then so are the elements on the left-hand side.
Similarly, since $v_i p_i v_i^* = v_i v_i^* = r_i$ we deduce that:
\begin{equation}
E_{12}^{(n)} \otimes \sigma_k(r_k) = \tilde{\pi}(e_{12}) \left( E_{21}^{(n)} \otimes \sigma_k(p_k) \right) \tilde{\pi}(e_{12}),
\end{equation}
so $E_{12}^{(n)} \otimes \sigma_k(r_k)$ are in $M_0$ as well for $k=1,2$.
 Adding and subtracting the elements we have obtained so far from $\tilde{\pi}(e_{12})$, we deduce that $E_{12}^{(n)} \otimes 1_{\tilde{N}} \otimes  1_{\Matrix_2} \in M_0$.

We now turn to recover the standard element $E_{13}^{(n)}$. Towards that, denote  $C =E_{12}^{(2)}+  E_{21}^{(2)}$ and note that: 
\begin{align}\label{eq:recall-e13}
     E_{13}^{(n)} \otimes 1_{\tilde{N}}\otimes 1_{\Matrix_2}= \tilde{\pi}(e_{13})  + E_{13}^{(n)} \otimes p  \otimes 1_{\Matrix_2} -  E_{13}^{(n)} \otimes p  \otimes  C. 
\end{align}
We must therefore show that the last two summands are in $M_0$. 
To do so, we note that $\sigma_2(p_2)\leq p\otimes 1_{\Matrix_2}$, and as a result:
\begin{align}\label{eq: p_2 in E_1}
\tilde{\pi}(e_{13}) \left( E_{33}^{(n)} \otimes \sigma_2(p_2) \right) \tilde{\pi}(e_{31})  = E_{11}^{(n)} \otimes 1_{\tilde{N}_1} \otimes p_2 \otimes E_{11}^{(2)}.
\end{align}
This element is in $M_0$ by the previous claim, and thereby so is the element:
\begin{equation}\label{eq: D}
\left (E_{11}^{(n)} \otimes 1_{\tilde{N}_1} \otimes p_2 \otimes E_{11}^{(2)} \right ) \left ( E_{11}^{(n)} \otimes \sigma_1(p_1) \right ) = E_{11}^{(n)} \otimes p_1 \otimes p_2 \otimes E_{11}^{(2)}.
\end{equation}
Recalling the definition of $p$ in Eq. \ref{eq:def-of-p}, we can write: 
\begin{equation*}
    E_{11}^{(n)} \otimes p \otimes E_{11}^{(2)} = \left ( E_{11}^{(n)} \otimes 1_{\tilde{N}_1} \otimes p_2 \otimes E_{11}^{(2)} \right ) + \left ( E_{11}^{(n)} \otimes \sigma_1(p_1) \right ) - E_{11}^{(n)} \otimes p_1 \otimes p_2 \otimes E_{11}^{(2)}
\end{equation*}
and by the previous equations, \ref{eq: p_2 in E_1} and \ref{eq: D}, this element is in $M_0$ as well. In an analogous fashion, we  get that the elements $E_{11}^{(n)} \otimes p \otimes E_{22}^{(2)}$ and  $E_{11}^{(n)} \otimes p_1 \otimes p_2 \otimes E_{22}^{(2)}$ are in $M_0$; hence also $E_{11}^{(n)} \otimes p \otimes 1_{\Matrix_2}$  and $E_{11} \otimes p_1 \otimes p_2 \otimes 1_{\Matrix_2}$. We, therefore, get that the element
\[
\left( E_{11}^{(n)} \otimes p \otimes 1_{\Matrix_2} \right) \tilde{\pi}(e_{13}) = 
E_{13}^{(n)} \otimes p \otimes C
\]
 is in $M_0$. Looking at equation \ref{eq:recall-e13} again, we are only left to show that the element $E_{13}^{(n)} \otimes p \otimes 1_{\Matrix_2}$ is in $M_0$. 
 
 To that end, we note that $\sigma_1(p_1),\sigma_2(p_2)\leq p\otimes 1_{\Matrix_{2}}$ and therefore:
 \begin{equation}
\begin{split}\label{eq: p_2+p_1 in E_13}
    \tilde{\pi}(e_{13}) & \left[ E_{31}^{(n)} \otimes(\sigma_1(p_1)+\sigma_2(p_2))\right]\tilde{\pi}(e_{13})  \\ & = E_{13}^{(n)} \otimes \left(1_{\tilde{N}_1} \otimes p_2 \otimes E_{11}^{(2)} 
    + p_1 \otimes 1_{\tilde{N}_2} \otimes E_{22}^{(n)}\right).
\end{split}
\end{equation}
Second, using the definition of $\tilde{\pi}(f_{ij})$, we see that
\begin{equation}\label{eq: p1 times p2 in E_13}
    E_{13}^{(n)} \otimes p_1 \otimes p_2 \otimes 1_{\Matrix_2} = \tilde{\pi}(f_{13}) (E_{11} \otimes p_1 \otimes p_2 \otimes 1_{\Matrix_2})
\end{equation}
The elements in equations \ref{eq: p_2+p_1 in E_13} and \ref{eq: p1 times p2 in E_13}, together with  $E_{13}^{(n)} \otimes (\sigma_1(p_1) +\sigma_2(p_2))$ are sufficient to express $E_{13}^{(n)} \otimes p \otimes 1_{\Matrix_2}$, and we conclude that the latter is in $M_0$.
\end{proof}

\begin{Rmk} \label{rmk:p'v}
We note, for future use, that during the proof of Claim \ref{claim:Mn is in M0} we in fact showed that the elements $E_{12}^{(n)} \otimes \sigma_1(r_1)$, $E_{12}^{(n)} \otimes \sigma_1(v_1)$ and $E_{11}^{(n)} \otimes p \otimes E_{11}^{(2)}$ are in $M_0$, and by the conclusion of the same claim, we get the same elements where $E_{12}^{(n)}$ is replaced with any $E_{ij}^{(n)}$.
Symmetrically, the elements $E_{ij}^{(n)} \otimes \sigma_2(r_2)$, $E_{ij}^{(n)} \otimes \sigma_2(v_2)$ and $E_{ij}^{(n)} \otimes p \otimes E_{22}^{(2)}$ are also in $M_0$.
\end{Rmk}
Notice that due to Claim \ref{claim:Mn is in M0}, the goal of showing that $M_0 = \tilde{M}$ is reduced to showing that $E_{11}^{(n)} \otimes \tilde{N} \otimes  \Matrix_2 \subseteq M_0$. This allows us to focus from now on, mainly on the $(1,1)$-entry of $\Mn$.
We proceed by showing that  $\tilde{N}_1$ and $\tilde{N}_2$ can be generated separately on the diagonal blocks of the $2$ by $2$ matrices, that is:  
\begin{Claim}\label{claim:isolating N's}
$M_0$ contains $\Mn\otimes  \sigma_1(\tilde{N}_1)$ and $\Mn \otimes \sigma_2(\tilde{N}_2)$. 
\end{Claim}


\begin{proof}

First, note that it suffices to generate $\Mn\otimes \sigma_1(N_1)$ (and symmetrically $\Mn\otimes \sigma_2(N_2)$). Indeed, by Lemma \ref{lem:generating-factors-from-corners},  $N_1$ and $v_1$ generate $\tilde{N}_1$, or equivalently  $E_{11}^{(n)}\otimes \sigma_1(N_1)$ and $E_{11}^{(n)}\otimes \sigma_1(v_1)$ generate $E_{11}^{(n)}\otimes \sigma_1(\tilde{N}_1)$;  
by Remark \ref{rmk:p'v},
the element $E_{11}^{(n)} \otimes\sigma_1(v_1)$ is in $M_0$, and so together with Claim \ref{claim:Mn is in M0}, we get all of $\Mn\otimes \sigma_1(\tilde{N}_1)$.

In order to generate $\Mn\otimes \sigma_1(N_1)$, define the following non-unital *-representation
$$\rho:A \to  \tilde{M}:\quad  x\mapsto  \pi_1(x) \otimes 1_{\tilde{N}_2} \otimes E_{11}^{(2)} + \pi_2(x) \otimes 1_{\tilde{N}_1} \otimes E_{22}^{(2)} $$
Alternatively, $\rho$ may be viewed as a unital *-representation of $A$ into the von Neumann (non-unital) subalgebra $\Mn \otimes \left( \sigma_1(N_1) \oplus \sigma_2(N_2) \right) \subseteq \tilde{M}$; we will show that this subalgebra is generated by the image of $\rho$. 
Note that $\rho$ is a direct sum representation, naturally equivalent to the direct sum $\pi_1 \oplus \pi_2$. Moreover, the factor representations $\pi_1$ and  $\pi_2$  may be assumed to be  disjoint, for otherwise the corresponding extreme traces $\varphi_1$ and $\varphi_2$ would be identical, and the original trace $\varphi$ would have been extreme to begin with. Thus, by \cite[Proposition 5.2.4]{dixmier1977c} or by \cite[Proposition 6.B.4]{bekka2020unitary}, we have:
\begin{align*}
    \rho(A)''
    &=\left( \pi_1(A)'' \otimes 1_{\tilde{N}_2} \otimes E_{11}^{(2)} \right) \oplus \left( \pi_2(A)'' \otimes 1_{\tilde{N}_1} \otimes E_{22}^{(2)} \right) \\
    &=\left( \Mn\otimes N_1\otimes 1_{\tilde{N}_2}\otimes E_{11}^{(2)}\right) \oplus  \left( \Mn\otimes N_2\otimes 1_{\tilde{N}_1}\otimes E_{22}^{(2)} \right)\\
    &= \left(\Mn\otimes \sigma_1(N_1) \right) \oplus \left( \Mn\otimes \sigma_2(N_2) \right)
\end{align*}
It is therefore left to explain why $\rho(A)$ is contained in $M_0$. Write:
\begin{align*}
    \rho(e_{ij})+E_{ij}\otimes\sigma_1(p_1)+E_{ij}\otimes \sigma_2(p_2) &= E_{ij}\otimes 1_{\tilde{N}}\otimes 1_{\Matrix_2} \\
    \rho(f_{ij})+E_{ij}\otimes \sigma_1(p_1)+E_{ij}\otimes \sigma_2(p_2) &= \tilde{\pi}(f_{ij})
\end{align*}
We have thus expressed $\rho(e_{ij})$ and $\rho(f_{ij})$ as combinations of elements which have already been shown to belong to $M_0$ in Claims \ref{claim:tilde-p is in M0} and \ref{claim:Mn is in M0}.  It follows that $\rho(A)$ is contained in $M_0$.    
\end{proof}

Now, we start mixing the two algebras $\tilde{N_1
}$ and $\tilde{N}_2$ together.

\begin{Claim}\label{claim:generating N1 tensor N2}
$M_0$ contains the algebras $\Mn \otimes \tilde{N} \otimes E_{11}^{(2)}$ and $\Mn \otimes \tilde{N}_1 \otimes E_{22}^{(2)}$.

\end{Claim}

\begin{proof}
The proof for both algebras in the claim is the same, so we will argue for the first algebra. Denote $w=p \otimes (E_{12}^{(2)} + E_{21}^{(2)}) $. Then 
 \[
 E_{11}^{(n)} \otimes w =  \left ( E_{11}^{(n)} \otimes p \otimes 1_{\Matrix_2} \right ) \tilde{\pi}(e_{13}) \left( E_{31}^{(n)} \otimes 1_{\tilde{N}} \otimes 1_{\Matrix_2} \right)
 \]
 This element is in $M_0$  because the first multiplicand is in $M_0$ due to Remark \ref{rmk:p'v}, while the third one is in $M_0$ due to Claim \ref{claim:Mn is in M0}. 
 
By Claim \ref{claim:isolating N's}, for every $n_2 \in \tilde{N}_2$, the element  $E_{11}^{(n)}\otimes \sigma_2(n_2)$ is in $M_0$. Now, 
\[
E_{11}^{(n)}\otimes \Ad_w \left( \sigma_2(n_2) \right) = 
E_{11}^{(n)}\otimes \left( 1_{N_1} \otimes p_2n_2p_2 + p_1 \otimes n_2 \right) \otimes E_{11}^{(2)} 
\]
and since $p_1$ is orthogonal to $1_{N_1}$, 
\[
E_{11}^{(n)}\otimes \left[ \sigma_1(p_1) \cdot \Ad_w \left(  \sigma_2(n_2) \right) \right] = E_{11}^{(n)}\otimes p_1 \otimes n_2 \otimes E_{11}^{(2)}. 
\]
Elements of this form (all of which are in $M_0$) constitute the algebra $E_{11}^{(n)}\otimes p_1\otimes \tilde{N}_2\otimes E_{11}^{(2)}$, which together with the algebra $E_{11}^{(n)}\otimes \tilde{N}_1 \otimes 1_{\tilde{N}_2} \otimes E_{11}^{(2)}$ generate the algebra $E_{11}^{(n)}\otimes \tilde{N}_1 \otimes \tilde{N}_2 \otimes E_{11}^{(2)}$ by Lemma \ref{lem:generate tensors}.
Applying Claim \ref{claim:Mn is in M0}, we conclude that $\Mn \otimes \tilde{N}_1 \otimes \tilde{N}_2 \otimes E_{11}^{(2)}$ is contained in $M_0$.
\end{proof}

We are now ready to show that $\tilde{\pi}(A)''=M_0$  is equal to $\tilde{M}=\Mn\otimes \tilde{N}\otimes \Matrix_2$, and thus to complete the proof of Theorem \ref{Second Theorem}.

Retaining the notation $w$ from the previous claim we note that:
\begin{align*}
\left( E_{11}^{(n)}\otimes w \right) \left( E_{11}^{(n)} \otimes 1_{\tilde{N}} \otimes E_{11}^{(2)} \right) 
= E_{11}^{(n)} \otimes p \otimes E_{12}^{(2)}
\end{align*}
Together with the element $E_{11}^{(n)} \otimes p \otimes E_{11}^{(2)}$ we conclude that the algebra $E_{11}^{(n)} \otimes p \otimes \Matrix_2$ is contained in $M_0$, since $\Matrix_2$ is generated by $\{ E_{11}^{(2)}, E_{12}^{(2)} \}$. This algebra, together with the algebra $E_{11}^{(n)} \otimes \tilde{N} \otimes 1_{\Matrix_2}$ (which is in $M_0$ by Claim \ref{claim:generating N1 tensor N2}), generate $E_{11}^{(n)} \otimes \tilde{N} \otimes \Matrix_2$ due to Lemma \ref{lem:generate tensors}. Using Claim \ref{claim:Mn is in M0} we conclude that $M_0 = \tilde{M}$.

\subsection{Factorizable quantum channels}\label{subsec:quantum}
Fix $n\in \N$, and let $\mathcal{FM}(n)$ denote the set of all factorizable quantum channels $T:\Mn\to \Mn$. Then $\mathcal{FM}(n)$ is a compact and convex subset of a finite-dimensional Euclidean space.  In \cite{musat2021factorizable} it is shown how factorizable quantum channels are parameterized by traces on the \csalg $\MnMn$. 
Explicitly, a continuous, surjective and affine map $\Phi: \Traces{\Mn * \Mn} \rightarrow \mathcal{FM}(n)$ is constructed. Using this map, we deduce the statement from the introduction:

\begin{proof}[Proof of Corollary \ref{cor:quantum-channels}] Let $T\in \mathcal{FM}(n)$ and choose $\varphi \in \Phi^{-1}(T)$.  Then by Theorem \ref{Second Theorem}, there exists a sequence of extreme traces $\varphi_k$ that converges to $\varphi$. Consider finite factors $(M_k,\tau_k)$  and representations $\pi_k:\MnMn\to M_k$ with generating images, such that $\varphi_k=\tau_k\circ\pi_k$. Let $\alpha_k,\beta_k$ denote the restriction of $\pi_k$ to the first and second $\Mn$ copies. Then $T_k:=\Phi(\varphi_k)$ factorizes surjectively through $M$ via the maps $\alpha_k,\beta_k$.  Continuity of $\Phi$ implies that the sequence $T_k$  converges to $T$.
\end{proof}

\section{Variety of finite-dimensional representations} \label{section:findim}
In this section we discuss approximations of finite-dimensional traces by extreme traces of arbitrarily high dimension.
Such approximations are needed for the proof of Theorem \ref{Second Theorem}, since we have to dilate the algebra $M$ using a small projection in order to invoke Lemma \ref{lem:embedding N as a corner}. The finite-dimensionality allows us to  use basic tools from algebraic geometry.

Let $V$ be a real algebraic variety. Then we may consider $V$ both
with the Zariski topology and with the real topology induced from
the standard topology on $\R$. We will need the following.
\begin{Lemma}
\label{lem:real-alg-geometry}Let $V$ be an irreducible real algebraic
variety, and let $U$ be a non-empty Zariski-open subset of $V$.
Then $U$ is dense in $V$ in the real topology.
\end{Lemma}

\begin{proof}
Since $V$ is irreducible, any non-empty Zariski-open subset is Zariski-dense,
and thus (by general topological considerations) is irreducible as well. This in particular applies to
an affine chart of $V$. As $V$ is covered by affine charts,  there is no loss of generality in assuming
that $V$ is affine to begin with. We thus identify $V\subseteq\R^{n}$
with the zero locus of a set of polynomials in $\R[x_{1},...,x_{n}]$.
There is also no loss of generality in assuming that $U$ is a basic
open set of $V$, that is, we assume $U=V\backslash f^{-1}(0)$ for
some polynomial $f\in\R[x_{1},...,x_{n}]$ which is non-zero on $V$. 

Assume by contradiction that $U$ is not dense in the real topology
on $V$. Then there exists a set $W$ which is open in the real topology
on $V$, and which is contained in $f^{-1}(0)$. Take a non-empty intersection of $W$ with an irreducible component of $f^{-1}(0)$. The set of regular
points of this irreducible component of the algebraic variety $f^{-1}(0)$ is a subvariety of the same dimension \cite[Ch. 2, Theorem 4]{onishchik1993lie}, thus, in particular, $W$ contains a smooth
point $x\in W$. This implies that the dimension of $f^{-1}(0)$ at
$x$ is the same as the dimension of $V$ at $x$. Thus $f^{-1}(0)$
is a closed proper algebraic subvariety of $V$ of the same dimension. But
this contradicts the irreducibility of $V$. 
\end{proof}

We will apply this lemma to representation varieties of certain \csalg. 
Let $\C^*\langle x_1,...,x_n \rangle$ be the free $*$-algebra on $n$ generators. Its elements may be considered as polynomials in the $2n$ non-commuting variables $x_1,...,x_n$ and $x_1^*,...,x_n^*$.  Let $A$ be a unital $*$-algebra with a finite generating set $S$ of size $n$, and fix an ordering $S=(s_1,...,s_n)$. The assignment $x_i\mapsto s_i$ extends to an isomorphism $\C^*\langle x_1,...,x_n \rangle/ I\cong A$  for some two-sided ideal $I$.
Fix $k\in \N$, and view $\Matrix_k(\C)\cong \C^{k^2}$ as a \emph{real} algebraic variety.  
 The collection $V_k=\mathrm{Hom}(A,\Matrix_k)$ of all unital $*$-representations $A \to \Matrix_k$ may be identified with the subset of the direct product variety $\Matrix_k^n$ consisting of all $y=(y_1,...,y_n)\in \Matrix_k^n$ for which $p(y)=0$ for $p\in I$. Since these relations of matrices are polynomials in the entries, we can view $V_k$ as the set of zeros of a collection of polynomials in commuting variables. Note that this collection can be taken to be finite by Noetherianity. This gives $V_k$ a structure of an  affine real algebraic variety\footnote{we note that it is not a complex variety since the $*$-operation is not algebraic}. We consider the Zariski topology on $V_k$, as well as the real topology coming from $\C^{k^2}$. The latter is the topology of pointwise convergence i.e. $\pi_i\to \pi$ iff $\pi_i(x)\to \pi(x)$ for all $x\in A$. 

\begin{Prop}\label{prop:surj-morphisms-zariski-open}
Let $A$ be a $*$-algebra and let $S=(s_1,...,s_n)\subseteq A$ be a finite generating set. Fix $k\in \N$ and let $V_k$ denote the variety of all unital $*$-homomorphisms $A \to \Matrix_k$ as defined above. Then the subset $V_k^\mathrm{sur}\subseteq V_k$ consisting of all surjective *-representations is open in the Zariski topology. Assume moreover that $V_k$ is irreducible and that $V_k^\mathrm{sur}$ is non-empty.  Then $V_k^\mathrm{sur}$ is dense in the Zariski topology, as well as in the real topology. 
\end{Prop}

\begin{proof}
Let $\pi \in V_k^{\mathrm{sur}}$ (the statement clearly holds if $V_k^{\mathrm{sur}}$ is empty). Then each element of $\Matrix_k$ is of the form $p(\pi(S))=p(\pi(s_1),...,\pi(s_n))$ for some polynomial $p\in \C^*\langle x_1,...,x_n \rangle$.
Stated differently, the set of all $m(\pi(S))$, where $m$ runs over the set of all monomials in $\C^*\langle x_1,...,x_n \rangle$, spans $\Matrix_{k}$. It follows that there exist $k^2$ monomials $m_1,...,m_{k^2}$ such that the set $\{m_1(\pi(S)),...,m_{k^2}(\pi(S))\}$ is a linear basis for $\Matrix_k$.
Let $U_{\pi}$ denote the set of all $\pi'\in V_k$ such that $\{m_1(\pi'(S)),...,m_{k^2}(\pi'(S))\}$ is linearly independent.
Then $U_\pi$ is a Zariski-open neighbourhood of $\pi$, as it is determined by the non-vanishing of the determinant. Moreover  $U_{\pi}\subseteq V_k^{\mathrm{sur}}$.
Repeating this for each $\pi\in V_k^{\mathrm{sur}}$, we conclude that $V_k^{\mathrm{sur}}=\bigcup_{\pi} U_\pi$ is Zariski-open, and in particular open in the real topology. 

Assume now that $V_k$ is irreducible. Then, inasmuch as  $V_k^{\mathrm{sur}}$ is not empty, it is dense in the real topology by Lemma \ref{lem:real-alg-geometry}, and in particular in the Zariski topology. 
\end{proof}
Let  $A$  be the universal \csalg generated (as a \csalg) by a finite set $S$ and some set of relations.  Let $A_0$ denote  the dense $*$-subalgebra generated by $S$. The universality property ensures that any $*$-representation of $A_0$ extends uniquely to a $*$-representation of $A$.
The space of $*$-representations $A\to \Matrix_k$ is therefore identified with the the space of $*$-representations $A_0\to \Matrix_k$. Moreover, since $\Matrix_k$ is finite-dimensional, a $*$-representation $A\to \Matrix_k$ is surjective if and only if its restriction to $A_0$ is surjective. 
Proposition \ref{prop:surj-morphisms-zariski-open}  therefore applies to the $C^*$-algebra $A$, that is, the set of $*$-representations $A\to \Matrix_k$ is a real algebraic variety, and the subset of surjective representations is open therein.  \\

Let $\mathrm{T}_{\mathrm{fd}}(A)$ denote the space of all finite-dimensional traces. $\mathrm{T}_{\mathrm{fd}}(A)$ is  always a face of $\Traces{A}$ \cite[Proposition 2.1]{musat2021factorizable}, but it is usually not closed \cite[Remark 2.4]{musat2021factorizable}, so in particular it is not a simplex. The following proposition shows that this face has dense extreme points, assuming certain properties on its representation variety, and as a result, $\overline{\mathrm{T}_{\mathrm{fd}}(A)}$ is a face of $\Traces{A}$ which is a Poulsen simplex.

\begin{Prop}\label{prop:finite-dimensional amplification}
Let $A$ be the universal \csalg generated by some finite set, and some relations. For $k\in \N$ denote by $V_k$ the variety of all $*$-representations $A\to \Matrix_k$.
Assume that for any $k\in \N$ for which  $V_k$ is non-empty, there exists infinitely many multiples $k'$ of $k$ for which $V_{k'}$ is an irreducible variety and contains a surjective representation.
Then for any $m\in \N$, the set of finite-dimensional extreme traces of dimension at least $m$ is dense in $\mathrm{T}_{\mathrm{fd}}(A)$. In particular, the face  $\overline{\mathrm{T}_{\mathrm{fd}}(A)}$ is a Poulsen simplex.
\end{Prop}

\begin{proof}
Let $\varphi$ be a finite-dimensional trace and assume first that it is extreme. Then there is  $k\in \N$ and a  unital $*$-representation $\pi:A\to \Matrix_k$, so that $\varphi=\tr_k \circ \pi$. Fix $m\in \N$. By assumption, there exists  $k'\geq m$ which is a multiple of $k$ such that $V_k$ is irreducible and includes a surjective homomrophism. Fix a unital $*$-embedding $\Matrix_k\hookrightarrow \Matrix_{k'}$, say the diagonal embedding, and denote its composition with $\pi$ by $\tilde{\pi}:A \to \Matrix_{k'}$. Note that $\tr_{k'}\circ \tilde{\pi} =\tr_k \circ \pi$. By Proposition \ref{prop:surj-morphisms-zariski-open},
there exists a sequence of surjective unital $*$-representations $\tilde{\pi}_i:A\to \Matrix_{k'}$ such that $\tilde{\pi}_i(x)\to \tilde{\pi}(x)$ for all $x\in A$. The corresponding traces $\varphi_i=\tr_{k'}\circ\tilde{\pi}_i$ are extreme, since $\tilde{\pi}_i(A)=\Matrix_{k'}$ is a factor, and they converge pointwise to $\varphi$.

Now, if $\varphi$ is finite-dimensional but perhaps not extreme, it can still be written as a finite convex combination of extreme finite-dimensional traces. Moreover, for our density statement, we may assume that the coefficients of this convex combination are rational.
It follows once again that $\varphi=\tr_k\circ\pi$ for some $k\in \N$, and some (perhaps non-surjective) unital $*$-representation $\pi:A\to \Matrix_k$, see for example the argument in  \cite[Lemma 3.7]{hadwin2018stability}. We can then proceed as in the above paragraph.
\end{proof}

\begin{Cor}\label{cor: finite-dimensional reps of Mn are poulsen}
Suppose $A$ is either $C^*(F_d)$ for $2\leq d < \infty$, or a $d$-fold free product $\Mn * ... * \Mn$ with $n\geq 3, d \geq 2$.  Then for any $m\in \N$, the set of finite-dimensional extreme traces of dimension at least $m$ is dense in $\mathrm{T}_{\mathrm{fd}}(A)$. In particular, the face $\overline{\mathrm{T}_{\mathrm{fd}}(A)}$ is the Poulsen simplex.
\end{Cor}

\begin{proof}
We will show that in each case $A$ satisfies the conditions of Proposition \ref{prop:finite-dimensional amplification}. 

In the case where $A=C^*(F_d)$ the representation variety $V_k$ is clearly the product $U(k)^d$ where $U(k)$ is the real algebraic group of unitary matrices of dimension $d$.  This space is connected and therefore irreducible \cite[\S1.2]{borel2012linear}. Moreover, it consists of a surjective representation since $\Matrix_k$ is generated by two unitaries. 

In the case where $A=\Mn * ... * \Mn$, the representation variety is a product $V_k={V'_k}^d$ where $V'_k$ is the representation variety of $\Mn$. Assume that $n$ divides $k$ for otherwise $V_k=V_k'=\emptyset$.  
Consider the real algebraic action of $U(k)$ on $V'_k$ by conjugation. This action is transitive, as any two representations are conjugate \cite[Lemma III.2.1]{davidson1996c}. Thus $V'_k$ is a quotient of an irreducible variety, and hence irreducible in itself. Therefore $V_k$ is also irreducible.
As long as $n$ divides $k$, $V^\mathrm{sur}_k$ is non-empty. Indeed, $\Matrix_k$ is generated by two unitaries which gives rise to a surjective $*$-homomorphism $\Mn * ... * \Mn \to \Matrix_n(\Matrix_{k/n})\cong \Matrix_k$ - see Lemma \ref{lem:reps of Mn, genearotrs and relations} and the paragraph following it. 
\end{proof}

\begin{Rmk}\label{rmk:approximating a single extremal finite-dimensional}
The first part of the proof of Proposition \ref{prop:finite-dimensional amplification} shows the following. Let $A = \Mn * ... * \Mn$ and let $\varphi$ be a finite-dimensional extreme trace of dimension $k$. Then for every multiple $k'$ of $k$, one can approximate $\varphi$ by a sequence of extreme traces $\varphi_n$, all of which are of dimension $k'$. This fact is used in the proof of Theorem \ref{Second Theorem} in \S\ref{sec:matrix algebras}.
\end{Rmk}

\section{Faces of the trace simplex}\label{sec:faces}

A closer look at the proofs of Theorem \ref{Thm:free products of commutative algebras}  and Theorem \ref{Second Theorem} shows that the constructed approximating traces of the algebra $A$ satisfy the following more refined property.\\

{\bf Refined approximation property:} Let $\varphi_1$ and $\varphi_2$ be two extreme traces on the unital \csalg $A$, and let $M_1$ and $M_2$ denote  their corresponding von Neumann algebras. Then the trace $\varphi =\frac{1}{2}(\varphi_1+\varphi_2)$ can be approximated arbitrarily well by a trace $\varphi'$ whose von Neumann algebra is an amplification of  the tensor product $M_1\otimes M_2$.\\

We recall that an amplification of a von Neumann algebra $M$ is any von Neumann algebra of the form $p\Mn(M)p$ where $n\in \N$ and $p\ne 0$ is a projection of $\Mn(M)$.

Given a unital \csalg $A$ and a property $\mathcal{P}$ of finite factors we denote 
\[
    \TracesP{A}{\PP}= \mathrm{conv}\{ \varphi\in \Traces{A}: M_\varphi \text{ is a finite factor satisfying } \PP\}
\]
where $M_\varphi$ denotes the von Neumann algebra (defined uniquely up to isomorphism) associated with $\varphi$. It is not hard to see that the closed convex hull of a collection of extreme points in a metrizable simplex is always a closed face and is a metrizable simplex as well.  As such, the closure $\overline{\TracesP{A}{\PP}}$ is a closed face of $\Traces{A}$, and is a simplex.

We are particularly interested in properties $\PP$ of finite factors which are closed under the following operations:
\begin{enumerate}
    \item \label{item:her-corner} If $M$ satisfies $\PP$ then so does any amplification of $M$.
    \item \label{item:her-tensor} If $M_1, M_2$ satisfy $\PP$ then so does their tensor product $M_1\otimes M_2$.
\end{enumerate}

\begin{Prop} \label{prop:poulsen-face}
    Let $A$ be a unital \csalg satisfying the refined approximation property. Let $\PP$ be a property of finite factors which is closed under taking tensor products and amplifications. Then the closed face $\overline{\TracesP{A}{\PP}}$ is a Poulsen simplex.
\end{Prop}

\begin{proof}
    By applying Lemma \ref{lem:approximating-sums-of-two-extreme-points} to $\overline{\TracesP{A}{\PP}}$ we see that it is enough to approximate traces of the form $\varphi=\frac{1}{2}(\varphi_1+\varphi_2)$ where $\varphi_1,\varphi_2\in \TracesP{A}{\PP} $ are extreme traces. The fact that $\overline{\TracesP{A}{\PP}}$ is a face implies that $\varphi_1$ and $\varphi_2$ are also extreme when considered as points in the whole simplex $\Traces{A}$. Now, the von Neumann algebras $M_1$ and $M_2$ associated with $\varphi_1$ and $\varphi_2$ are in $\mathcal{P}$. The refined approximation property ensures that $\varphi$ may be approximated by a trace $\varphi'$ whose von Neumann algebra is an amplification  $N$ of $M_1\otimes M_2$. Note that $N$ is a factor and so $\varphi'$ is an extreme trace. As $\PP$ is closed under tensor products and amplifications,  it contains $N$, and so $\varphi'\in \TracesP{A}{\PP}$.
\end{proof}

\begin{Cor}\label{cor:Poulsen-faces}
 Consider a \csalg of the form $A=A_0*A_1*A_2$ where $A_1$ and $A_2$ are \csalgs of continuous functions on compact metrizable spaces without isolated points, and $A_0$ is any unital separable \csalg. Let $\PP$ be a property of finite factors which is closed under taking tensor products and amplifications. Then the closed face $\overline{\TracesP{A}{\PP}}$ of $\Traces{A}$ is a Poulsen simplex.
 The same is true when $A$ is taken to be $\MnMn$ with $n\geq 4$.
\end{Cor}

Here are a few particularly interesting examples of properties of finite factors which are closed under amplifications and tensor products, thus giving rise to certain faces which are Poulsen. For simplicity, we restrict our attention to von Neumann algebras on separable Hilbert spaces.
\begin{enumerate}
    \item Let $\PP=\mathrm{fd}$ be the property of being finite-dimensional. The face $\TracesP{A}{\mathrm{fd}}$ is the space of all finite-dimensional traces, previously discussed in \S\ref{section:findim}. 
    \item Let $\PP$ be the property of von Neumann algebras being \emph{amenable} (see \cite{anantharaman2017introduction}). By \cite[Theorem 3.2.2]{brown06}, $\TracesP{A}{\PP}$ coincides with the set of \emph{uniformly amenable traces} (see \cite[Def. 3.2.1]{brown06}). 
    \item A trace $\varphi$ on $A$ is called \emph{amenable} if it admits an extension to a state on $\BH$, for some embedding $A\subseteq \BH$, such that $\varphi(uTu^*)=\varphi(T)$ for all unitary elements $u\in A$ and operators $T\in \BH$ \cite[Def. 3.1.1]{brown06}. For \csalgs satisfying the \emph{local lifting property} (see \cite{pisier2020tensor}), a trace $\varphi$ is amenable if and only if its von Neumann algebra is Connes-embeddable, i.e there exists an embedding of $\pi_\varphi(A)''$ into the ultrapower $\mathcal R^\omega$ of the hyperfinite $\twoone$-factor $\mathcal R$ \cite[Prop. 6.3.4]{brown06}. Commutative \csalgs and nuclear \csalgs satisfy the local lifting property \cite[Cor. 7.12]{pisier2020tensor} and so also free products of such \cite[Theorem. 9.44]{pisier2020tensor}. Thus, for such algebras, the space of amenable traces coincides with $\TracesP{A}{\PP}$ where $\PP$ is the property of being Connes-embeddable.
\end{enumerate}
Other interesting properties of finite factors which are closed under tensor products and amplifications include McDuff, Property (T), Haagerup, Properly proximal, Gamma, weak amenability, and weak exactness. For the sake of conciseness, we omit the list of definitions and proofs that these properties are indeed closed in the appropriate sense.\footnote{We thank Cyril Houdayer for suggesting these properties.} 

\section{Poulsen simplex obstructions}\label{sec:obstructions}
In this last section, we mention a few necessary conditions for the trace space of a group to be a Poulsen simplex. With certain modifications, similar conditions may be stated for \csalgs. We will denote here the set of traces on a group $G$ by $\Tr{G}$. The space of extreme traces, i.e. characters, as $\Ch{G} = \partial_e \Tr{G}$, as is common for groups.

\subsection{Property (ET)}
The obstruction we formulate is based on the following group property:

\begin{Def}
A group $G$ is said to have property (ET) if the trivial character $1_G$ is isolated in $\Ch{G}$.
\end{Def}
Clearly, finite groups have property (ET). More generally, any group with Kazhdan's property (T) has property (ET); this is easily seen when recalling that one of the equivalent definitions of property (T) is that the trivial representation is isolated in the unitary dual with the Fell topology \cite[Theorem 1.2.5]{bekka_harp}. Other types of examples are ones of the form $\mathrm{SL}_2(\Z[1/p])$ for $p$ prime, see \cite{levit2023characterlimits}.

We start by observing that property (ET) is inherited by quotients:

\begin{Lemma} \label{lem:quotient of ET}
If a group $G$ has property (ET) then so does any quotient of $G$.
\end{Lemma}

\begin{proof}
Let $q:G\to H$ be  a quotient map between groups. For any trace $\varphi$ on $H$, the pullback $\varphi \circ q $ is a trace on $G$.  Moreover, it is not hard to see that if $\varphi$ is a character then so is $\varphi\circ q$. Thus,  if $\varphi_n$ are characters of $H$ converging to the trivial character $1_H$, then $\varphi_n \circ q$ are characters of $G$ converging to the trivial trace $1_G=1_H\circ q$.  Since $G$ has property (ET) the sequence $\varphi_n\circ q$ eventually stabilizes on $1_G$. It follows that $\varphi_n$ eventually stabilizes on $1_H$.   
\end{proof}

\begin{Lemma}\label{lem:property-ET-traces}
Let $G$ be a group with property (ET). Let $\varphi_n$ be a sequence of traces on $G$, and let $\mu_n$ be the corresponding Borel probability measures on $\Ch{G}$ whose barycenter is $\varphi_n$. If $\varphi_n$ converge to the trivial character $1_G$, then $\mu_n(\{1_G\})\to 1$.
\end{Lemma}

\begin{proof}
We recall that traces take values in the unit disk of $\C$. The assumption that $G$ has property (ET) implies that there exists $\eps>0$, and a finite set $F\subseteq G$ such that for all $1\ne \chi\in \Ch{G}$:
\[
    \max_{g\in F}(1- \Re(\chi(g)))\geq \eps
\]
Therefore
\begin{align*}
    \max_{g\in F}|1-\varphi_n(g)| &=\max_{g\in F}\lvert\int_{\Ch{G}} 1-\chi(g) d\mu_n(\chi)\rvert \\ &\geq
    \max_{g\in F}\Re \left( \int_{\Ch{G}} 1-\chi(g) d\mu_n(\chi)\right) \\
    &\geq \eps \cdot \mu_n(\chi\in \Ch{G}\mid \chi\ne 1)
\end{align*}
The statement follows as the left hand side converges to $0$.
\end{proof}

In contrast to property (T), property (ET) passes to free products:
\begin{Prop}\label{prop:property-ET-free-products}
If the groups $G_1,...,G_m$ have property (ET) then so does their free product $G=G_1*...*G_m$.
\end{Prop}
\begin{proof}
By induction, it suffices to prove the statement for $m=2$.
Assume by contradiction  we have characters $\varphi_n\in \Ch{G}$ that converge to $1_G$, and yet none of them is equal to $1_G$.
For $i=1,2$, the restrictions $\varphi_n\mid_{G_i}$ are traces on $G_i$, and by the Lemma \ref{lem:property-ET-traces}, we have that $\lim_n \mu_{i,n}(\{1_{G_1}\})= 1$ where $\mu_{i,n}$ is the probability measure supported on $\Ch{G_i}$ whose barycenter is $\varphi_n\restrict{G_i}$.
It follows that for a sufficiently large $n$, which we now fix, both $\mu_{1,n}(\{1_{G_1}\}), \mu_{2,n}(\{1_{G_2}\})$ are strictly greater than $\frac{1}{2}$. We denote $\varphi=\varphi_n$. We will reach a contradiction by showing that $\varphi=1_G$. 

Let $\pi:G\to \mathcal{U}(\mathcal{H})$ be a tracial representation corresponding to $\varphi$. Let $p_i:\mathcal{H}\to \mathcal{H}^{G_i}$ denote the projection onto the subspace of $G_i$-fixed vectors. Then $p_i$ is clearly in $\pi(G_i)''$ and in particular in the von Neumann algebra $M:=\pi(G)''$.   By the above, we have that $\tau_M(p_1),\tau_M(p_2)> \frac{1}{2}$. Denote $p=p_1\wedge p_2$, that is, the maximal projection which is contained in both $p_1$ and $p_2$. Then $p$ is the projection onto the subspace of all the $G$-fixed vectors in $\mathcal{H}$, and it is in particular in the center of $M$.  By Kaplansky's formula \cite[Proposition 2.4.5]{anantharaman2017introduction}
\[
    \tau_M(p) \geq \tau(p_1)+\tau(p_2)-\tau(1_M) >0,
\]
and so $p\ne 0$.
But $M$ is a factor, hence  $p=1_M$ and so $\varphi=1_G$.
\end{proof}

Let $H$ be a group endowed with homomorphisms  $f_1:H\to G_1$ and $f_2:H\to G_2$ into two other groups $G_1$ and $G_2$. The amalgamated product  $G_1*_H G_2$ is the quotient of the free product of $G_1$ and $G_2$ by the normal subgroup generated by elements of the form $f_1(h)f_2(h)^{-1}$ for $h\in H$. Amalgamated products of finite groups constitute a rich class of virtually free groups. Using Proposition \ref{prop:property-ET-free-products} together with Lemma \ref{lem:quotient of ET} we immediately get: 
\begin{Cor}\label{cor: amalgam-ET}
If the groups $G_1$ and $G_2$ have property (ET), then any amalgamated product of the form $G_1*_H G_2$ has property (ET) as well. 
\end{Cor}

Next, we consider a relative setting of traces and characters. For a normal subgroup $N\leq G$ we denote by $\relTr{G}{N}:=\Tr{N}^G$ the set of traces on $N$ which are fixed by the conjugation action of $G$. This is a compact and convex set, and we denote by $\relCh{G}{N}=\partial_e\relTr{G}{N}$ its set of extreme points. Clearly, $1_N\in \relCh{G}{N}$.

\begin{Prop}\label{prop:property-ET-normal-subgroups}
Let $N$ be a normal subgroup of $G$. If $N$ has property (ET), then $1_N$ is an isolated point of $\relCh{G}{N}$
\end{Prop}
\begin{proof}
Let $\varphi\in \relCh{G}{N}$. Viewing  $\varphi$ as a trace on $N$, we denote by $\mu$ the corresponding probability measure supported on $\Ch{N}$ whose barycenter is $\varphi$. As $\varphi$ varies and approaches $1_N\in \Tr{N}$, the value $\alpha:=\mu(\{1_G\})$ tends to $1$ - this is the content of Lemma \ref{lem:property-ET-traces}. In particular, we may restrict to a sufficiently small neighbourhood of $1_N\in \Tr{N}$ so that $\alpha$ is positive. It is left to show that in that case $\varphi$ must be equal to $1_N$. This is clearly true if $\alpha=1$. Otherwise, we may write  $\varphi=\alpha 1_N +(1-\alpha)\psi$ where  $\psi:=(1-\alpha)^{-1}(\varphi -\alpha 1_N)\in \relTr{G}{N}$. The extremity of $\varphi$ implies that $\psi$, and therefore also $\varphi$, is equal to $1_N$.
\end{proof}

\subsection{The topology of the space of extreme points}
We recall that the set of extreme points of a metrizable Poulsen simplex is homeomorphic to an infinite-dimensional separable Banach space by \cite{los_78}. We immediately obtain: 
\begin{Prop}\label{prop:non-connected->not-Poulsen}
Let $A$ be a unital separable \csalg and assume that the set of extreme points of $\Traces{A}$ is not connected. Then $\Traces{A}$ is not a Poulsen simplex.
\end{Prop}

We deduce the following obstructions. 
\begin{Thm}\label{thm:free product of ET is not Poulsen}
Let $G_1,...,G_m$ be groups with property (ET), and let $G=G_1*...*G_m$ be their free product. Then $\Tr{G}$ is not a Poulsen simplex. The same is true if $G$ is  an amalgamated free product of the form $G=G_1*_H G_2$. 
\end{Thm}

\begin{proof}
By Proposition \ref{prop:property-ET-free-products} (or Corollary \ref{cor: amalgam-ET} in the second case)  $G$ has property (ET). Thus $\Ch{G}$ admits an isolated point and so $\Tr{G}$ is not Poulsen by Proposition \ref{prop:non-connected->not-Poulsen}. 
\end{proof}
Theorem \ref{thm:free product of ET is not Poulsen} in particular implies Theorem \ref{thm:intro amalgam not poulsen} from the introduction. We finish with the following additional obstruction applies.

\begin{Prop}\label{prop:normal-subgroup ET -> not Poulsen}
If a group $G$ admits a non-trivial normal subgroup with property (ET), then $\Tr{G}$ is not a Poulsen simplex.
\end{Prop}
\begin{proof}
Suppose $G$ admits such a normal subgroup $N$.
The restriction of $\varphi\in \Ch{G}$ to $N$ is an element of $\relCh{G}{N}$ \cite[Lemma 14]{thoma1964unitare}. We thus get a continuous map $r:\Ch{G}\to \relCh{G}{N}$, and this map is moreover surjective \cite[Lemma 16]{thoma1964unitare}. By Proposition \ref{prop:property-ET-normal-subgroups}, $1_N$ is an isolated point of $\relCh{G}{N}$. Since $N$ is non-trivial, $1_N$ is not the only element of $\relTr{G}{N}$ (e.g $\delta_e$), and thus by Krein-Milman, it is not the only element of $\relCh{G}{N}$. It follows that the preimage of $1_N$ under the map $r$ is a non-trivial clopen subset of $\Ch{G}$, and so $\Tr{G}$ is not Poulsen by Proposition \ref{prop:non-connected->not-Poulsen}. 
\end{proof}

\bibliography{Poulsen}

\newcommand{\etalchar}[1]{$^{#1}$}
\begin{thebibliography}{BKKO17}

\bibitem[AGV14]{abert2014kesten}
Mikl{\'o}s Ab{\'e}rt, Yair Glasner, and B{\'a}lint Vir{\'a}g.
\newblock Kesten’s theorem for invariant random subgroups.
\newblock 2014.

\bibitem[AL07]{aldous_lyons}
David Aldous and Russell Lyons.
\newblock Processes on unimodular random networks.
\newblock {\em Electronic Journal of Probability}, 12:1454--1508, 2007.

\bibitem[AP17]{anantharaman2017introduction}
Claire Anantharaman and Sorin Popa.
\newblock An introduction to ii1 factors.
\newblock {\em preprint}, 8, 2017.

\bibitem[BBH21]{bader2021charmenability}
Uri Bader, R{\'e}mi Boutonnet, and Cyril Houdayer.
\newblock Charmenability of higher rank arithmetic groups.
\newblock {\em arXiv preprint arXiv:2112.01337}, 2021.

\bibitem[BBHP22]{bader2022charmenability}
Uri Bader, R{\'e}mi Boutonnet, Cyril Houdayer, and Jesse Peterson.
\newblock Charmenability of arithmetic groups of product type.
\newblock {\em Inventiones mathematicae}, pages 1--57, 2022.

\bibitem[BdLH20]{bekka2020unitary}
Bachir Bekka and Pierre de~La~Harpe.
\newblock {\em Unitary representations of groups, duals, and characters},
  volume 250.
\newblock American Mathematical Soc., 2020.

\bibitem[Bek07]{bekka2007operator}
Bachir Bekka.
\newblock Operator-algebraic superridigity for $sl_n(\mathbb{Z}), n \geq 3$.
\newblock {\em Inventiones mathematicae}, 169(2):401--425, 2007.

\bibitem[BF20]{bekka2020characters}
Bachir Bekka and Camille Francini.
\newblock Characters of algebraic groups over number fields.
\newblock {\em arXiv preprint arXiv:2002.07497}, 2020.

\bibitem[BH21]{boutonnet2021stationary}
R{\'e}mi Boutonnet and Cyril Houdayer.
\newblock Stationary characters on lattices of semisimple lie groups.
\newblock {\em Publications math{\'e}matiques de l'IH{\'E}S}, pages 1--46,
  2021.

\bibitem[BKKO17]{breuillard2017c}
Emmanuel Breuillard, Mehrdad Kalantar, Matthew Kennedy, and Narutaka Ozawa.
\newblock C*-simplicity and the unique trace property for discrete groups.
\newblock {\em Publications math{\'e}matiques de l'IH{\'E}S}, 126(1):35--71,
  2017.

\bibitem[Bla85]{blackadar1985shape}
Bruce Blackadar.
\newblock Shape theory for c*-algebras.
\newblock {\em Mathematica Scandinavica}, pages 249--275, 1985.

\bibitem[BLH20]{bekka_harp}
M~Bachir Bekka and Pierre~de La~Harpe.
\newblock Unitary representations of groups, duals, and characters.
\newblock {\em (No Title)}, 2020.

\bibitem[Bor12]{borel2012linear}
Armand Borel.
\newblock {\em Linear algebraic groups}, volume 126.
\newblock Springer Science \& Business Media, 2012.

\bibitem[Bow15]{bowen15}
Lewis Bowen.
\newblock Invariant random subgroups of the free group.
\newblock {\em Groups, Geometry, and Dynamics}, 9(3):891--916, 2015.

\bibitem[Bro06]{brown06}
Nathanial~Patrick Brown.
\newblock {\em Invariant means and finite representation theory of $
  C^{*}$-algebras}, volume~13.
\newblock American Mathematical Soc., 2006.

\bibitem[BV22]{bader-iti}
Uri Bader and Itamar Vigdorovich.
\newblock Charmenability and stiffness of arithmetic groups.
\newblock {\em arXiv preprint arXiv:2208.07347}, 2022.

\bibitem[Cho80]{choi80}
Man~Duen Choi.
\newblock The full $ c^{*}$-algebra of the free group on two generators.
\newblock {\em Pacific J. Math}, 87(1):41--48, 1980.

\bibitem[CMP19]{collins2019automorphism}
Beno{\^\i}t Collins, Michael Magee, and Doron Puder.
\newblock Automorphism-invariant positive definite functions on free groups.
\newblock {\em arXiv preprint arXiv:1906.01518}, 2019.

\bibitem[Cun82]{cuntz1982k}
Joachim Cuntz.
\newblock The k-groups for free products of c*-algebras.
\newblock {\em Operator algebras and applications, Part}, 1:81--84, 1982.

\bibitem[Dav96]{davidson1996c}
Kenneth~R Davidson.
\newblock {\em C*-algebras by example}, volume~6.
\newblock American Mathematical Soc., 1996.

\bibitem[Dix77]{dixmier1977c}
Jacques Dixmier.
\newblock $c^*$-algebras, volume 15 of north-holland mathematical library.
\newblock {\em North-Holland Publishing Company}, 47:49--56, 1977.

\bibitem[Dix11]{dixmier2011neumann}
Jacques Dixmier.
\newblock {\em von Neumann algebras}, volume~27.
\newblock Elsevier, 2011.

\bibitem[EL92]{exel1992finite}
Ruy Exel and Terry~A Loring.
\newblock Finite-dimensional representations of free product c*-algebras.
\newblock {\em International Journal of Mathematics}, 3(04):469--476, 1992.

\bibitem[Gel15]{gelander2015lecture}
Tsachik Gelander.
\newblock A lecture on invariant random subgroups.
\newblock {\em arXiv preprint arXiv:1503.08402}, 2015.

\bibitem[HKB23]{hartman2017stationary}
Yair Hartman, Mehrdad Kalantar, and Uri Bader.
\newblock Stationary c*-dynamical systems.
\newblock {\em Journal of the European Mathematical Society (EMS Publishing)},
  25(5), 2023.

\bibitem[HM11]{haagerup2011factorization}
Uffe Haagerup and Magdalena Musat.
\newblock Factorization and dilation problems for completely positive maps on
  von neumann algebras.
\newblock {\em Communications in Mathematical Physics}, 303(2):555--594, 2011.

\bibitem[How77]{howe1977representations}
Roger Howe.
\newblock On representations of discrete, finitely generated, torsion-free,
  nilpotent groups.
\newblock {\em Pacific Journal of Mathematics}, 73(2):281--305, 1977.

\bibitem[HS18a]{hadwin2018stability}
Don Hadwin and Tatiana Shulman.
\newblock Stability of group relations under small hilbert--schmidt
  perturbations.
\newblock {\em Journal of Functional Analysis}, 275(4):761--792, 2018.

\bibitem[HS18b]{hadwin2018tracial}
Don Hadwin and Tatiana Shulman.
\newblock Tracial stability for c*-algebras.
\newblock {\em Integral Equations and Operator Theory}, 90(1):1--35, 2018.

\bibitem[Jam65]{jamison1965eigenvalues}
Benton Jamison.
\newblock Eigenvalues of modulus 1.
\newblock {\em Proceedings of the American Mathematical Society},
  16(3):375--377, 1965.

\bibitem[JNV{\etalchar{+}}21]{mip=re}
Zhengfeng Ji, Anand Natarajan, Thomas Vidick, John Wright, and Henry Yuen.
\newblock Mip*=re.
\newblock {\em Communications of the ACM}, 64(11):131--138, 2021.

\bibitem[KKN21]{kaluba2021property}
Marek Kaluba, Dawid Kielak, and Piotr~W Nowak.
\newblock On property (t) for $\text{Aut}(f_n)$ and $\text{SL}_n(\mathbb{Z})$.
\newblock {\em Annals of Mathematics}, 193(2):539--562, 2021.

\bibitem[KR86]{kadison1986fundamentals}
Richard~V Kadison and John~R Ringrose.
\newblock {\em Fundamentals of the theory of operator algebras. Volume II:
  Advanced theory}.
\newblock Academic press New York, 1986.

\bibitem[KS22]{kennedy2022noncommutative}
Matthew Kennedy and Eli Shamovich.
\newblock Noncommutative choquet simplices.
\newblock {\em Mathematische Annalen}, 382(3-4):1591--1629, 2022.

\bibitem[LL23]{lavi2023characters}
Omer Lavi and Arie Levit.
\newblock Characters of the group eld (r) for a commutative noetherian ring r.
\newblock {\em Advances in Mathematics}, 419:108948, 2023.

\bibitem[LOS78]{los_78}
Joram Lindenstrauss, Gunnar Olsen, and Yaki Sternfeld.
\newblock The poulsen simplex.
\newblock {\em Annales de l'institut Fourier}, 28(1):91--114, 1978.

\bibitem[LS04]{lubotzky_shalom}
Alexander Lubotzky and Yehuda Shalom.
\newblock Finite representations in the unitary dual and ramanujan groups.
\newblock {\em Contemporary Mathematics}, 347:173--190, 2004.

\bibitem[LSV23]{levit2023characterlimits}
Arie Levit, Raz Slutsky, and Itamar Vigdorovich.
\newblock Spectral gap and character limits in arithmetic groups.
\newblock {\em arXiv preprint arXiv:2308.05562}, 2023.

\bibitem[Lub96]{lubotzky1996free}
Alexander Lubotzky.
\newblock Free quotients and the first betti number of some hyperbolic
  manifolds.
\newblock {\em Transformation groups}, 1(1):71--82, 1996.

\bibitem[LV23]{levit2022characters}
Arie Levit and Itamar Vigdorovich.
\newblock Characters of solvable groups, hilbert--schmidt stability and dense
  periodic measures.
\newblock {\em Mathematische Annalen}, pages 1--49, 2023.

\bibitem[MR20]{musat2020non}
Magdalena Musat and Mikael R{\o}rdam.
\newblock Non-closure of quantum correlation matrices and factorizable channels
  that require infinite dimensional ancilla (with an appendix by narutaka
  ozawa).
\newblock {\em Communications in Mathematical Physics}, 375(3):1761--1776,
  2020.

\bibitem[MR21]{musat2021factorizable}
Magdalena Musat and Mikael R{\o}rdam.
\newblock Factorizable maps and traces on the universal free product of matrix
  algebras.
\newblock {\em International Mathematics Research Notices},
  2021(23):17951--17970, 2021.

\bibitem[MVN37]{murray1937rings}
Francis~J Murray and John Von~Neumann.
\newblock On rings of operators. ii.
\newblock {\em Transactions of the American Mathematical Society},
  41(2):208--248, 1937.

\bibitem[Nit20]{nitsche2020computer}
Martin Nitsche.
\newblock Computer proofs for property {(T)}, and {SDP} duality.
\newblock {\em arXiv preprint arXiv:2009.05134}, 2020.

\bibitem[OVM93]{onishchik1993lie}
Arkadi~L Onishchik, Ernest~B Vinberg, and V~Minachin.
\newblock {\em Lie groups and Lie algebras}.
\newblock Springer, 1993.

\bibitem[Oza13]{ozawa2013tsirelson}
Narutaka Ozawa.
\newblock Tsirelson's problem and asymptotically commuting unitary matrices.
\newblock {\em Journal of Mathematical Physics}, 54(3):032202, 2013.

\bibitem[Pet14]{peterson}
Jesse Peterson.
\newblock Character rigidity for lattices in higher-rank groups.
\newblock {\em Preprint.(http://www. math. vanderbilt. edu/peters10/rigidity.
  pdf)}, 2014.

\bibitem[Phe01]{phelps2001lectures}
Robert~R Phelps.
\newblock {\em Lectures on Choquet’s theorem}.
\newblock Springer, 2001.

\bibitem[Pis20]{pisier2020tensor}
Gilles Pisier.
\newblock {\em Tensor Products of C*-Algebras and Operator Spaces: The
  Connes--Kirchberg Problem}, volume~96.
\newblock Cambridge University Press, 2020.

\bibitem[Pou61]{Poulsen}
Ebbe~T Poulsen.
\newblock A simplex with dense extreme points.
\newblock {\em Annales de l'institut Fourier}, 11:83--87, 1961.

\bibitem[PT16]{peterson2016character}
Jesse Peterson and Andreas Thom.
\newblock Character rigidity for special linear groups.
\newblock {\em Journal f{\"u}r die reine und angewandte Mathematik (Crelles
  Journal)}, 2016(716):207--228, 2016.

\bibitem[Puk74]{pukanszky1974characters}
L~Pukanszky.
\newblock Characters of connected lie groups.
\newblock {\em Bulletin of the American Mathematical Society}, 80(4):709--712,
  1974.

\bibitem[R{\o}r02]{rordam2002classification}
Mikael R{\o}rdam.
\newblock Classification of nuclear, simple c*-algebras.
\newblock In {\em Classification of nuclear C*-algebras. Entropy in operator
  algebras}, pages 1--145. Springer, 2002.

\bibitem[Sak12]{sakai2012c}
Sh{\^o}ichir{\^o} Sakai.
\newblock {\em $C^*$-algebras and $W^*$-algebras}.
\newblock Springer Science \& Business Media, 2012.

\bibitem[Sri08]{srivastava2008course}
Sashi~Mohan Srivastava.
\newblock {\em A course on Borel sets}, volume 180.
\newblock Springer Science \& Business Media, 2008.

\bibitem[Str21]{strung2021introduction}
Karen~R Strung.
\newblock {\em An Introduction to C*-algebras and the Classification Program}.
\newblock Springer, 2021.

\bibitem[Tho64a]{thoma1964unzerlegbaren}
Elmar Thoma.
\newblock Die unzerlegbaren, positiv-definiten klassenfunktionen der
  abz{\"a}hlbar unendlichen, symmetrischen gruppe.
\newblock {\em Mathematische Zeitschrift}, 85(1):40--61, 1964.

\bibitem[Tho64b]{thoma1964unitare}
Elmar Thoma.
\newblock {\"U}ber unit{\"a}re darstellungen abz{\"a}hlbarer, diskreter
  gruppen.
\newblock {\em Mathematische Annalen}, 153(2):111--138, 1964.

\bibitem[VDN92]{voiculescu1992free}
Dan~V Voiculescu, Ken~J Dykema, and Alexandru Nica.
\newblock {\em Free random variables}.
\newblock Number~1. American Mathematical Soc., 1992.

\bibitem[Win18]{winter2018structure}
Wilhelm Winter.
\newblock Structure of nuclear c*-algebras: from quasidiagonality to
  classification and back again.
\newblock In {\em Proceedings of the International Congress of Mathematicians
  (ICM 2018) (In 4 Volumes) Proceedings of the International Congress of
  Mathematicians 2018}, pages 1801--1823. World Scientific, 2018.

\bibitem[Yos51]{Yoshizawa}
Hisaaki Yoshizawa.
\newblock Some remarks on unitary representations of the free group.
\newblock {\em Osaka Mathematical Journal}, 3(1):55--63, 1951.

\end{thebibliography}
\bibliographystyle{alpha}

\end{document}